\documentclass{article}

\usepackage{amsmath}
\usepackage{amssymb} 

\newtheorem{theorem}{Theorem}[section]
\newtheorem{lemma}[theorem]{Lemma}

\newtheorem{corollary}[theorem]{Corollary}

\newcommand{\mdl}[1]{\mathcal{#1}}  
\newcommand{\na}[1]{\mathit{#1}}    
\newcommand{\limplies}{\rightarrow}
\newcommand{\ex}[1]{\exists #1 \;} 
\newcommand{\fa}[1]{\forall #1 \;} 
\newcommand{\st}{ \; | \; } 

\newcommand{\NN}{\mathbb{N}}
\newcommand{\QQ}{\mathbb{Q}}
\newcommand{\ZZ}{\mathbb{Z}}
\newcommand{\RR}{\mathbb{R}}
\newcommand{\CC}{\mathbb{C}} 
\newcommand{\TT}{\mathbb{T}} 
\newcommand{\lip}{\langle}
\newcommand{\rip}{\rangle}
\newcommand{\la}{(}
\newcommand{\ra}{)}

\newenvironment{proof}{\topsep=\smallskipamount \partopsep=0pt
 \begin{trivlist} \itemindent=0pt
\item[\hskip \labelsep\emph{Proof.}]}{\unskip\nobreak\hfil\penalty50
   \quad\hbox{}\nobreak\hfil $\square$
   \parfillskip=0pt \finalhyphendemerits=0 \par \end{trivlist}\medskip}

\begin{document}

\title{Local stability of ergodic averages\footnote{Work by the first
    author partially supported by NSF grant DMS-0401042. Work by the
    second author partially supported by a postdoctoral grant from the
    Villum Kann Rasmussen Foundation.}}

\author{Jeremy Avigad, Philipp Gerhardy, and Henry Towsner}

\maketitle

\begin{abstract}
  We consider the extent to which one can compute bounds on the rate
  of convergence of a sequence of ergodic averages. It is not
  difficult to construct an example of a computable Lebesgue-measure
  preserving transformation of $[0,1]$ and a characteristic function
  $f = \chi_A$ such that the ergodic averages $A_n f$ do not converge
  to a computable element of $L^2([0,1])$. In particular, there is no
  computable bound on the rate of convergence for that sequence. On
  the other hand, we show that, for any nonexpansive linear operator
  $T$ on a separable Hilbert space, and any element $f$, it is
  possible to compute a bound on the rate of convergence of $\la A_n f
  \ra$ from $T$, $f$, and the norm $\| f^* \|$ of the limit. In
  particular, if $T$ is the Koopman operator arising from a computable
  ergodic measure preserving transformation of a probability space
  $\mdl X$ and $f$ is any computable element of $L^2(\mdl X)$, then
  there is a computable bound on the rate of convergence of the
  sequence $\la A_n f \ra$.

  The mean ergodic theorem is equivalent to the assertion that for
  every function $K(n)$ and every $\varepsilon > 0$, there is an $n$
  with the property that the ergodic averages $A_m f$ are stable to
  within $\varepsilon$ on the interval $[n,K(n)]$. Even in situations
  where the sequence $\la A_n f \ra$ does not have a computable limit,
  one can give explicit bounds on such $n$ in terms of $K$ and $\| f
  \| / \varepsilon$. This tells us how far one has to search to find
  an $n$ so that the ergodic averages are ``locally stable'' on a
  large interval. We use these bounds to obtain a similarly explicit
  version of the pointwise ergodic theorem, and show that our bounds
  are qualitatively different from ones that can be obtained using
  upcrossing inequalities due to Bishop and Ivanov.

  Finally, we explain how our positive results can be viewed as an
  application of a body of general proof-theoretic methods falling
  under the heading of ``proof mining.''
\end{abstract}

\section{Introduction}

Let $T$ be a nonexpansive linear operator on a Hilbert space $\mdl H$,
that is, a linear operator satisfying $\| T f \| \leq \| f \|$ for all
$f \in \mdl H$. For each $n \geq 1$, let $S_n f = f + T f + \ldots +
T^{n-1} f$ denote the sum of first $n$ iterates of $T$ on $f$, and let
$A_n f = \frac{1}{n} S_n f$ denote their average. The von Neumann mean
ergodic theorem asserts that the sequence $\la A_n f \ra$ converges in
the Hilbert space norm. The most important instance of the mean
ergodic theorem occurs when $\mdl H$ is the space $L^2(\mdl X)$ of
square-integrable functions on a probability space $\mdl X = \la X,
\mdl B, \mu\ra$, and $T$ is the Koopman operator $Tf = f \circ \tau$
associated to a measure preserving transformation, $\tau$, of that
space.  In that setting, the Birkhoff pointwise ergodic theorem
asserts that the sequence $\la A_n f \ra$ converges pointwise almost
everywhere, and in the $L^1$ norm, for any $f$ in $L^1(\mdl X)$. The
transformation $\tau$ is said to be \emph{ergodic} if there are no
nontrivial invariant subsets, in the sense that $\tau^{-1}(A) = A$
implies $\mu(A) = 0$ or $\mu(A) = 1$. In that case, $\lim_{n \to
  \infty} A_n f$ is a.e.~equal to the constant function $\int f \;
d\mu$, which is to say, almost every sequence of time averages
converges to the average over the entire space.

It is known that, in general, the sequence $\la A_n f \ra$ can
converge very slowly. For example, Krengel \cite{krengel:78} has shown
that for any ergodic automorphism of the unit interval under Lebesgue
measure, and any sequence $(a_n)$ of positive reals converging to
$0$, no matter how slowly, there is a subset $A$ of the interval
such that, if $\chi_A$ denotes the characteristic function of $A$, then
\[
\lim_{n \to \infty} \frac{1}{a_n}| (A_n \chi_A)(x) - \mu(A)| = \infty
\]
almost everywhere, and
\[
\lim_{n \to \infty} \frac{1}{a_n}\| A_n \chi_A - \mu(A)\|_p = \infty
\]
for every $p \in [1,\infty)$. For related results and references, see
\cite[Section 0.2]{kachurovskii:96} and \cite[notes to Section
  1.2]{krengel:85} for related results and references.) Here, however,
  we will be concerned with the extent to which a bound on the rate of
  convergence can be computed from the initial data. That is, given
  $\mdl H$, $T$, and $f$ in the statement of the von Neumann ergodic
  theorem, we can ask whether it is possible to compute, for each
  rational $\varepsilon > 0$, a value $r(\varepsilon)$, such that for
  every $n$ greater or equal to $r(\varepsilon)$, we have $\|
  A_{r(\varepsilon)}f - A_n f \| < \varepsilon$.

Determining whether such an $r$ is computable from the initial data is
not the same as determining its rate of growth. For example, if $\la
a_n \ra_{n \in \NN}$ is any computable sequence of rational numbers
that decreases monotonically to $0$, then a rate of convergence can be
computed trivially from the sequence: given $\varepsilon$, one need
only run through the elements of the sequence and until one of them
drops below $\varepsilon$. On the other hand, it is relatively easy to
construct a computable sequence $\la b_n \ra$ of rational numbers that
converge to $0$, for which there is \emph{no} computable bound on the
rate of convergence. It is also relatively easy to construct a
computable, monotone, bounded sequence $\la c_n \ra$ of rationals that
does not have a computable limit, which implies that there is no
computable bound on the rate of convergence of this sequence,
either. These examples are discussed in
Section~\ref{computability:section}.

Indeed, in Section~\ref{computability:section}, we show that there are
a computable Lebesgue-measure preserving transformation of the unit
interval $[0,1]$ and a computable characteristic function $f = \chi_A$
such that the limit of the sequence $\la A_n f\ra$ is not a computable
element of $L^2([0,1])$. For this we rely on standard notions of
computability for Hilbert spaces, which we review there. The
noncomputability of the limit implies, in particular, that there is no
computable bound on the rate of convergence of $\la A_n f \ra$. On the
other hand, we show that, for any nonexpansive linear operator $T$ on
a separable Hilbert space and any element $f$, one can compute a bound
on the rate of convergence of $\la A_n f \ra$ from $T$, $f$, and the
norm $\| f^* \|$ of the limit.  In particular, if $T$ is the Koopman
operator arising from a computable ergodic measure preserving
transformation of a probability space $\mdl X$ and $f$ is any
computable element of $L^2(\mdl X)$, then $\| f^* \|$ is equal to $|
\int f \; d\mu |$, which is computable; and hence there is a
computable bound on the rate of convergence.

In situations where the rate of convergence of the ergodic averages is
not computable from $T$ and $f$, is there any useful information to be
had? The logical form of a statement of convergence provides some
guidance. The assertion that the sequence $\la A_n f \ra$ converges
can be represented as follows:
\begin{equation}
\label{eq:a}
\fa {\varepsilon > 0} \ex n \fa {m \geq n} (\| A_m f - A_n f \| <
\varepsilon).
\end{equation}
A bound on the rate of convergence is a function $r(\varepsilon)$ that
returns a witness to the existential quantifier for each $\varepsilon
> 0$. It is the second universal quantifier that leads to
noncomputability, since, in general, there is no finite
test that can determine whether a particular value of $n$ is large
enough.  But, classically, the statement of convergence is equivalent
to the following:
\begin{equation}
\label{eq:b}
\fa {\varepsilon > 0, M : \NN \to \NN } \ex n 
(M(n) \geq n \limplies \| A_{M(n)} f - A_n f \| < \varepsilon).
\end{equation}
To see this, note that if, for some $\varepsilon > 0$, the existential
assertion in (\ref{eq:a}) is false, then for every $n$ there is an $m
\geq n$ such that $\| A_m f - A_n f \| \geq \varepsilon$. In that case,
$\varepsilon$ together with any function $M(n)$ that returns such an
$m$ for each $n$ represents a counterexample to
(\ref{eq:b}). Assertion (\ref{eq:a}) is therefore equivalent to the
statement that there is no such counterexample, i.e.~assertion
(\ref{eq:b}). 

But if the space is explicitly presented and (\ref{eq:b}) is true,
then for each $\varepsilon > 0$ and $M$ one can compute a witness to
the existential quantifier in (\ref{eq:b}) by simply trying values of
$n$ until one satisfying $\| A_{M(n)} - A_n \|$ is found. Thus,
(\ref{eq:b}) has an inherent computational interpretation. In
particular, given any function $K(n)$, suppose we apply (\ref{eq:b})
to a function $M(n)$ which, for each $n$, returns a value $m$ in the
interval $[n,K(n)]$ maximizing $\| A_m f - A_n f \|$. In that case,
(\ref{eq:b}) asserts
\[
\fa {\varepsilon > 0} \ex n \fa{m \in [n,K(n)]} \| A_m f - A_n f \|
< \varepsilon.
\]
In other words, if $r(\varepsilon)$ is a function producing a
witness to the existential quantifier, then, rather than computing an
absolute rate of convergence, $r(\varepsilon)$ provides, for each
$\varepsilon > 0$, a value $n$ such that the ergodic averages $A_m f$
are stable to within $\varepsilon$ on the interval $[n,K(n)]$.

It is now reasonable to ask for an explicit bound on $r(\varepsilon)$,
expressed in terms of in terms of $K$, $T$, $f$, and $\varepsilon$. In
Section~\ref{constructive:met:section}, we obtain bounds on
$r(\varepsilon)$ that, in fact, depend only on $K$ and $\rho = \lceil
\| f \| / \varepsilon \rceil$. Since the bound on the rate of
convergence is clearly monotone with $\rho$, our results show that,
for fixed $K$, the bounds are uniform on any bounded region of the
Hilbert space and independent of $T$. As special cases, we have the
following:
\begin{itemize}
\item If $K = n^{O(1)}$, then $r(f, \varepsilon) = 2^{2^{O(\rho^2
      \log \log \rho)}}$.
\item If $K = 2^{O(n)}$, then $r(f, \varepsilon) = 2^1_{O(\rho^2)}$,
  where $2_n^x$ denotes the $n$th iterate of $y \mapsto 2^y$ starting
  with $x$.
\item If $K = O(n)$ and $T$ is an isometry, then $r(f, \varepsilon)
  = 2^{O(\rho^2 \log \rho)}$.
\end{itemize}
Fixing $\rho$ and a parameterized class of functions $K$, one
similarly obtains information about the dependence of the bounds on
the parameters defining $K$.

In Section~\ref{pointwise:section}, we apply the results of
Section~\ref{constructive:met:section} to the case where $T$ is the
Koopman operator corresponding to a measure preserving transformation
on a probability space $\mdl X = \la X, \mdl B, \mu \ra$. The
pointwise ergodic theorem is equivalent to the assertion that for any
$f \in L^1(\mdl X)$, and for every $\lambda_1 > 0$, $\lambda_2 > 0$,
and $K$, there is an $n$ such that
\[
\mu(\{ x \st \max_{n \leq m \leq K(n)} | A_m f(x) - A_n f(x) | >
\lambda_1 \}) \leq \lambda_2.
\]
When $f$ is in $L^2(\mdl X)$, we provide explicit bounds on $n$ in
terms of $\lambda_1$, $\lambda_2$, $K$, and $\| f \|_2$. In this
setting, one can obtain similar bounds using alternative methods,
namely, from upcrossing inequalities due to Bishop and Ivanov. In
Section~\ref{upcrossing:section}, we show that the bounds extracted
using these methods are qualitatively different from the ones obtained
using the methods of Sections~\ref{constructive:met:section} and
\ref{pointwise:section}.

Our quantitative versions of the mean and pointwise ergodic theorems
are examples of Kreisel's no-counterexample interpretation
\cite{kreisel:51,kreisel:59alt}. Our extractions of bounds can be
viewed as applications of a body of proof theoretic results that fall
under the heading ``proof mining'' (see, for example,
\cite{kohlenbach:05,kohlenbach:unp,kohlenbach:unp:b}). What makes it
difficult to obtain explicit information from the usual proofs of the
mean ergodic theorem is their reliance on a nonconstructive principle,
namely, the assertion that any bounded increasing sequence of real
numbers converges. Qualitative features of our bounds---specifically,
the dependence only on $\| f \|$, $K$, and $\varepsilon$---are
predicted by the general metamathematical results of Gerhardy and
Kohlenbach \cite{gerhardy:kohlenbach:05}. Moreover, methods due to
Kohlenbach make it possible to extract useful bounds from proofs that
make use of nonconstructive principles like the one just
mentioned. These connections are explained in
Section~\ref{proof:mining:section}.

In the field of constructive mathematics, one is generally interested
in obtaining constructive analogues of nonconstructive mathematical
theorems. Other constructive versions of the ergodic theorems, due to
Bishop \cite{bishop:66,bishop:67,bishop:67b}, Nuber \cite{nuber:72},
and Spitters~\cite{spitters:06}, are discussed in
Sections~\ref{upcrossing:section} and
\ref{computability:section}. Connections to the field of reverse
mathematics are also discussed in Section~\ref{computability:section}.

The outline of this paper is as follows. In
Sections~\ref{constructive:met:section} and
Section~\ref{pointwise:section}, we provide our explicit versions of
the mean and pointwise ergodic theorems, respectively. In
Section~\ref{upcrossing:section}, we compare our results to similar
one obtained using upcrossing inequalities. In
Section~\ref{computability:section} we provide the general
computability and noncomputability results alluded to above. In
Section~\ref{proof:mining:section}, we explain the connections to
proof mining.

\emph{Acknowledgments.} We are grateful to Ulrich Kohlenbach and James
Cummings for comments and corrections.


\section{A quantitative mean ergodic theorem}
\label{constructive:met:section}

Given any operator $T$ on a Hilbert space and $n \geq 1$, let $S_n f =
\sum_{i < n} T^i f$, and let $A_n f = \frac{1}{n}S_n f$. The Riesz
version of the mean ergodic theorem is as follows.

\begin{theorem}
\label{met}
If $T$ is any nonexpansive linear operator on a Hilbert space and $f$
is any element, then the sequence $\la A_nf \ra$ converges.
\end{theorem}

We present a proof in a form that will be amenable to extracting a
constructive version.

\begin{proof}
  Let $M = \{ f \st T f = f \}$ be the subspace consisting of
  fixed-points of $T$, and let $N$ be the subspace generated by
  vectors of the form $u - T u$ (that is, $N$ is the closure of the
  set of linear combinations of such vectors).

  For any $g$ of the form $u - T u$ we have $\| A_n g \| = \frac{1}{n}
  \| u - T^n u \| \leq 2 \| u \| / n$, which converges to $0$. Passing
  to limits (using the fact that $A_n$ satisfies $\| A_n v \| \leq \|
  v \|$ for any $v$), we have that $A_n g$ converges to $0$ for every
  $g \in N$.

  On the other hand, clearly $A_n h = h$ for every $h \in M$. For
  arbitrary $f$, write $f = g + h$, where $g$ is the projection of $f$
  on $N$, and $h = f - g$. It suffices to show that $h$ is in
  $M$.  But we have
\begin{equation}
\label{riesz:eq}
\begin{split}
\| T h - h \|^2 & = \| T h \|^2 - 2 \lip Th, h \rip + \| h \|^2 \\
& \leq \| h \|^2 - 2 \lip Th, h \rip + \| h \|^2 \\
& = 2 \lip h, h \rip - 2 \lip Th, h \rip \\
& = 2 \lip h - Th, h \rip,
\end{split}
\end{equation}
and the right-hand side is equal to $0$, since $h$ is orthogonal to
$N$. So $Th = h$.
\end{proof}

The last paragraph of proof shows that $N^\bot \subseteq M$, and
moreover that $A_n f$ converges to $h$. It is also possible to show
that $M^\bot \subseteq N$, and hence $M = N^\bot$, which implies that
$h$ is the projection of $f$ on $M$. We will not, however, make use of
this additional information below.

As indicated in the introduction, the mean ergodic theorem is
classically equivalent to the following:
\begin{theorem}
\label{constructive:met}
Let $T$ and $f$ be as above and let $M: \NN \to \NN$ be any function
satisfying $M(n) \geq n$ for every $n$. Then for every $\varepsilon >
0$ there is an $n \geq 1$ such that $\| A_{M(n)}f - A_n f \| \leq
\varepsilon$.
\end{theorem}
Our goal here is to provide a constructive proof of this theorem that
provides explicit quantitative information. We will, in particular,
provide bounds on $n$ that depend only on $M$ and $\| f \| /
\varepsilon$.

For the rest of this section, we fix a nonexpansive map $T$ and an
element $f$ of the Hilbert space. A moment's reflection shows that
$A_n f$ lies in the cyclic subspace $\mdl H_f$ spanned by $\{ f, Tf,
T^2f, \ldots\}$, and so, in the proof of Theorem~\ref{met}, one can
replace $N$ by the subspace $N_f$ spanned by vectors of the form $T^i
f - T^{i+1} f$. Let $g$ be the projection of $f$ onto $N_f$. Then $g$
is the limit of the sequence $\la g_i \ra_{i \in \NN}$, where, for
each $i$, $g_i$ is the projection of $f$ onto the finite dimensional
subspace spanned by
\[
f-Tf, Tf- T^2f, \ldots, T^if-T^{i+1}f.
\]
The sequence $\la g_i \ra$ can be defined explicitly by
\[
g_0 = \frac{\lip f, f - Tf \rip}{\|f - Tf\|^2} (f-Tf),
\]
and
\[
g_{i+1}
= g_i + \frac{\lip f-g_{i}, T^i f - T^{i+1} f \rip}{\|T^i f -
  T^{i+1} f\|^2} (T^i f - T^{i+1} f).
\]
For each $i$, we can write $g_i = u_i - T u_i$, where the
sequence $\la u_i \ra_{i \in \NN}$ is defined by
\[
u_0 = \frac{\lip f, f - Tf \rip}{\|f - Tf\|^2} f,
\]
and
\[
u_{i+1}
= u_i + \frac{\lip f-g_i, T^i f - T^{i+1} f \rip}{\|T^i f -
  T^{i+1} f\|^2} T^i f.
\]
Note that this representation of $g_i$ as an element of the form $u -
Tu$ is not unique, since if $u$ and $u'$ differ by any fixed point of
$T$, $u - Tu = u' - Tu'$.

Finally, if we define the sequence $\la a_i \ra_{i \in \NN}$ by $a_i =
\| g_i \|$, then $\la a_i \ra$ is nondecreasing and converges to $\| g
\|$. We will see in Section~\ref{computability:section} that a bound
on the rate of convergence of $\la a_i \ra$ might not be computable
from $T$ and $f$. Our strategy here will be to show
that, given a fixed ``counterexample'' function $M$ as in the
statement of Theorem~\ref{constructive:met}, the fact that the
sequence $\la a_i \ra$ is bounded and increasing allows us to bound
the number of times that $M$ can foil our attempts to provide a
witness to the conclusion of the theorem.

First, let us record some easy but useful facts:

\begin{lemma} 
\label{4facts}
  \begin{enumerate}
  \item For every $n$ and $f$, $\| A_n f \| \leq \| f \|$.
  \item For every $n$ and $u$, $A_n (u - Tu) = (u - T^nu)/n$, and $\|
    A_n (u - Tu) \| \leq 2 \| u \| / n$.
  \item For every $f$, $g$, and $\varepsilon > 0$, if $\| f - g \| \leq
    \varepsilon$, then $\| A_n f - A_n g \| \leq \varepsilon$ for any
    $n$.
  \item For every $f$, if $\lip f, f - Tf \rip \leq
    \varepsilon$, then $\| T f - f \| \leq \sqrt{2 \varepsilon}$.
  \end{enumerate}
\end{lemma}

\begin{proof}
  The first two are straightforward calculations, the third follows
  from the first by the linearity of $A_n$, and the fourth follows
  from calculation (\ref{riesz:eq}) in the proof of Theorem~\ref{met},
  with $f$ in place of $h$.
\end{proof}

\begin{lemma}
\label{amn:bound:lemma}
  For every $f$, if $\| Tf - f \| \leq \varepsilon$, then for every $m
  \geq n \geq 1$ we have $\| A_m f - A_n f \| \leq (m-n) \varepsilon /
  2$. In particular, if $\| T f - f \| \leq \varepsilon$ and $m \geq 1$,
  then $\| A_m f - f \| \leq m \varepsilon /2$.
\end{lemma}

\begin{proof}
  Suppose $m \geq n \geq 1$. Then 
\begin{align*}
  \|A_m f - A_n f\| & = \| \frac{1}{m} \sum^{m-1}_{i=0} T^i f
  - \frac{1}{n} \sum^{n-1}_{j=0} T^j f\| \\  
  & = \frac{1}{mn} \|n \sum^{m-1}_{i=0} T^i f - m
  \sum^{n-1}_{j=0} T^j f \|\\
  & = \frac{1}{mn} \|n \sum^{m-1}_{i=n} T^i f - (m - n)
  \sum^{n-1}_{j=0} T^j f \|\\
\end{align*}
There are now $n \cdot (m-n)$ instances of $T^i f$ in the first term
and $n \cdot (m-n)$ instances of $T^j f$ in the second term. Pairing
them off and using that $\|T^i f - T^j f\| \leq (i-j) \cdot
\varepsilon$ for each such pair, we have
\begin{align*}
  \ldots & \leq \frac{1}{mn} \left(n \sum^{m-1}_{i=n} i - (m-n) 
    \sum^{n-1}_{j=0} j\right) \varepsilon \\
  & = \frac{1}{mn} \left(n \left(\frac {m(m-1)}{2} - \frac{n(n-1)}{2}\right) - (m-n)
    \left(\frac{n(n-1)}{2}\right)\right) \varepsilon \\
  & = \frac{1}{mn} \left(n \left(\frac {m(m-1)}{2}\right) - 
    m \left(\frac{n(n-1)}{2}\right)\right) \varepsilon \\
  & = (m-n) \varepsilon/2, 
\end{align*}
as required.
\end{proof}

We now turn to the proof of the constructive mean ergodic theorem
proper. The first lemma relates changes in $g_i$ to changes in $a_i$. 

\begin{lemma}
\label{aigi}
Suppose $|a_j - a_i| \leq \varepsilon^2 / (2
  \|f\|)$. Then $\| g_j - g_i \| \leq \varepsilon$.
\end{lemma}

\begin{proof}
  Assume, without loss of generality, $j > i$. Since
  $g_j$ is the projection of $f$ onto a bigger subspace, $g_j - g_i$
  is orthogonal to $g_i$. Thus, by the Pythagorean theorem, we have
\begin{align*}
\| g_j - g_i \|^2 & = \| g_j \|^2 - \| g_i \|^2 \\
& = | a_j^2 - a_i^2 | \\
& = | a_j - a_i | \cdot | a_j + a_i | \\
& \leq \frac{\varepsilon^2}{2 \|f\|} \cdot 2 \| f \| \\
& = \varepsilon^2,
\end{align*}
as required.
\end{proof}

The next lemma introduces a strategy that we will exploit a number of
times. Namely, we define an increasing function $F$ such that if, for
some $j$, $\| g_{F(j)} - g_j \|$ is sufficiently small, we have a
desired conclusion; and then argue that because the sequence $\la a_i
\ra$ is nondecreasing and bounded, sufficiently many iterations of
$F$ will necessarily produce such a $j$. (In the next lemma, we use
$F(j) = j+1$.)

\begin{lemma} \label{mainLem1} Let $\varepsilon > 0$, let $d =
  d(\varepsilon) = \lceil 32 \|f\|^4 / \varepsilon^4
  \rceil$. Then for every $i$ there is a $j$ in the interval $[i, i +
  d)$ such that $\| T(f-g_j) - (f-g_j)\| \leq \varepsilon$.
\end{lemma}

\begin{proof}
  By Lemma~\ref{4facts}.4, to obtain the conclusion, it suffices to
  ensure $\lip f - g_j, f - g_j - T(f - g_j) \rip \leq
  \varepsilon^2 / 2$. We have
\begin{align*}
\lip f - g_j, f - g_j - T(f - g_j) \rip & = \lip f -
g_j, f - T f \rip + \lip f - g_j, T g_j - g_j  \rip \\ 
& = \lip f - g_j, T g_j - g_j \rip
\end{align*}
because $g_j$ is the projection of $f$ on a space that includes $f -
T f$, and $f - g_j$ is orthogonal to that space. Recall that $g_j$ is
a linear combination of vectors of the form $T^k f - T^{k+1}
f$ for $k \leq j$, and $g_{j+1}$ is the projection of $f$ onto
a space that includes $T g_j - g_j$. Thus, continuing the calculation,
we have
\begin{align*}
\ldots & = \lip f - g_{j+1}, T g_j - g_j \rip + \lip
g_{j+1} - g_j, T g_j - g_j \rip \\ 
& = \lip g_{j+1} - g_j, g_j - T g_j  \rip \\
& \leq \| g_{j+1} - g_j \| \cdot \| T g_j - g_j \| \\ 
& \leq \| g_{j+1} - g_j \| (\| T g_j \| + \| g_j \|) \\ 
& \leq 2 \| g_{j+1} - g_j \| \cdot \| f \|
\end{align*}
Thus, if $\| g_j - g_{j+1} \| \leq \frac{\varepsilon^2}{4\|f\|}$, we have
the desired conclusion. 

Consider the sequence $a_i, a_{i+1}, a_{i+2}, \ldots, a_{i + d-1}$.
Since the $a_j$'s are increasing and bounded by $\| f \|$, for some $j
\in [i, i + d)$ we have $|a_{j+1} - a_j| \leq \frac{\| f \|}{d} \leq
\frac{\varepsilon^4}{32 \| f \|^3}$. By Lemma~\ref{aigi}, this implies
$\| g_j - g_{j+1} \|\leq \frac{\varepsilon^2}{4\|f\|}$, as required.
\end{proof}

\begin{lemma} 
\label{mainLem2} 
Let $\varepsilon > 0$, let $n \geq 1$, and let $d' = d'(n,\varepsilon)
= d(2 \varepsilon / n) = \lceil 2 n^4 \|f\|^4 / \varepsilon^4
\rceil$.  Then for any $i$, there is an $j$ in the interval $[i, i +
d')$ satisfying $\|A_n (f-g_j) - (f-g_j)\| \leq \varepsilon$.
\end{lemma}

\begin{proof}
  By the previous lemma, there is some $j$ in the interval $[i,i +
  d')$ such that $\|T(f - g_j) - (f - g_j)\| \leq 2
    \varepsilon / n$. By Lemma~\ref{amn:bound:lemma} this implies
  $\|A_n (f-g_j) - (f-g_j)\| \leq \varepsilon$.
\end{proof}

\begin{lemma}
  Let $\varepsilon > 0$, let $m \geq 1$, let $d'' = d''(m,\varepsilon) =
  d'(m,\varepsilon/2) = \lceil 32 m^4 \| f \|^4 / \varepsilon^4
  \rceil$. Further suppose $\| g_i - g_{i + d''} \|
  \leq \varepsilon / 4$. Then for any $n \leq m$, $\|A_n (f-g_i)
  - (f-g_i)\| \leq \varepsilon$.
\end{lemma}

\begin{proof}
  By the previous lemma, for any $n \leq m$, there is some $j$ in the
  interval $[i,i + d'')$ such that $\| A_n(f - g_j) - (f - g_j) \| \leq
  \varepsilon / 2$. This implies
  \begin{align*}
      \| A_n (f - g_i) - (f - g_i) \| & \leq \| A_n (f - g_i) - A_n (f -
      g_j) \|\\
      &  \quad \quad + \| A_n (f - g_j) - (f - g_j) \| + \| (f - g_j) -
      (f - g_i) \|\\ 
      & = \| A_n (g_j - g_i)  \| +  \| A_n (f - g_j) - (f - g_j) \| + 
        \| g_i - g_j \| \\
      & \leq \| A_n (f - g_j) - (f - g_j) \| + 2 \|g_j - g_i\| \\
      & \leq \varepsilon,
    \end{align*}
  since $\| g_j - g_i \| \leq \| g_{i + d''} - g_i \| \leq \varepsilon/4$.
\end{proof}

\begin{lemma}
\label{key:lemma}
Let $\varepsilon > 0$, let $m \geq 1$, let $d''' =
d'''(m,\varepsilon) = d''(m,\varepsilon/2) = \lceil 2^9 m^4 \| f
  \|^4 / \varepsilon^4 \rceil$. Further suppose $\| g_i - g_{i +
  d'''} \| \leq \varepsilon / 8$. Then for any $n
\leq m$, $\|A_m (f-g_i) - A_n(f-g_i)\| \leq \varepsilon$.
\end{lemma}
 
\begin{proof}
  Apply the previous lemma with $\varepsilon/2$ in place of
  $\varepsilon$. Then for every $n \leq m$, 
\begin{align*}
\|A_m (f-g_i) - A_n(f-g_i)\| & \leq \|A_m (f-g_i) - (f - g_i) \| +
\\
& \quad \quad \|A_n
(f-g_i) - (f - g_i) \| \\ 
&\leq \varepsilon/2 + \varepsilon / 2 = \varepsilon,
\end{align*}
as required.
\end{proof}

We will not need the following lemma until
Section~\ref{pointwise:section}, but it is a quick consequence of the
preceding result.

\begin{lemma}
\label{pet:key:lemma}
Let $\varepsilon > 0$, $m \geq 1$, $d''' = d'''(m,\varepsilon)$ as in
Lemma~\ref{key:lemma}, $e = \lceil \frac{2^7 \| f \|^2}{\varepsilon^2}
\rceil$, and $\hat d = \hat d (m,\varepsilon) = d''' \cdot e$. Then for
any $i$, there is a $j$ in the interval $[i,i + \hat d)$ such that for
every $n \leq m$, $\| A_n (f - g_j) - A_m (f - g_j) \| \leq
\varepsilon$.
\end{lemma}

\begin{proof}
  Consider the sequence $a_i, a_{i + d'''}, a_{i + d''' \cdot 2},
  \ldots, a_{i + d''' \cdot (e-1)}$. Since the sequence $\la a_i \ra$
  is increasing and bounded by $\| f \|$, for some $k < e$ and $j = i
  + d''' \cdot k$, we have $\| a_j - a_{j+d'''} \| \leq
  \varepsilon^2 / (2^7 \| f \|)$. By Lemma~\ref{aigi}, this implies
  $\| g_j - g_{j + d'''} \| \leq \varepsilon / 8$. Applying the
  previous lemma with $j$ in place of $i$, we have the desired
  conclusion.
\end{proof}

Let us consider where we stand. Given $\varepsilon > 0$ and a function
$M$ satisfying $M(n) \geq n$ for every $n$, our goal is to find an $n$ such
that $\| A_{M(n)} f - A_n f \| \leq \varepsilon$. Now, for any $n$ and
$i$, we have
\begin{align*}
  \|A_{M(n)} f - A_n f\| & = \|A_{M(n)}(f-g_i) + A_{M(n)} g_i - 
      (A_n(f-g_i) +
  A_n g_i)\| \\
  & \leq \|A_{M(n)} (f-g_i) - A_n (f-g_i)\| + \|A_{M(n)} g_i\| + \|A_n g_i\|.
\end{align*}
Lemma~\ref{key:lemma} tells us how to ensure that the first term on
the right-hand side is small: we need only find an $i$ such that $\|
g_{i + d'''} - g_i \|$ is small, for some $d'''$, depending on $M(n)$,
that is sufficiently large. On the other hand, by Lemma~\ref{4facts}
and $M(n) \geq n$, we have $\|A_n g_i\| \leq \|u_i\|/(2n)$ and
$\|A_{M(n)} g_i\| \leq \| u_i \| / (2 M(n)) \leq \| u_i \| / (2n)$.
Thus, to guarantee that the remaining two terms are small, it suffices
to ensure that $n$ is sufficiently large, in terms of $u_i$.

There is some circularity here: our choice of $i$ depends on $M(n)$,
and hence $n$, whereas our choice of $n$ depends on $u_i$, and hence
$i$. The solution is to define sequences $\la i_k \ra_{k \in \NN}$ and
$\la n_k \ra_{k \in \NN}$ recursively, as follows. Set $i_0 = 1$, and,
assuming $i_k$ has been defined, set
\begin{equation}
\label{nk:eq}
    n_k = \max(\left\lceil \frac{2 \|u_{i_k}\|}{\varepsilon} \right\rceil, 1)
\end{equation}
and
\begin{equation}
\label{ik:eq}
i_{k+1} = i_k + d'''(\varepsilon/2,M(n_k)) = i_k + \left\lceil \frac{2^{13}
  M(n_k)^4 \|f\|^4}{\varepsilon^4} \right\rceil
\end{equation}
Let $e = \lceil 2^9 \|f\|^2 / \varepsilon^2 \rceil$, and
consider the sequence $a_{i_0}, a_{i_1}, \ldots, a_{i_{e-1}}$.  Once
again, since this is increasing and bounded by $\| f \|$, for some $k
< e$ we have $| a_{i_{k+1}} - a_{i_k} | \leq \varepsilon^2 / 2^9 \| f
  \|$. Lemma~\ref{aigi} implies $\| g_{i_{k+1}} - g_{i_k} \| \leq
  \varepsilon / 16$. Write $i = i_k$ and $n = n_k$, so that
$i_{k+1} = i + d'''(M(n), \varepsilon/2)$. Applying
Lemma~\ref{key:lemma}, we have
\[
\| A_{M(n)} (f - g_i) - A_n (f - g_i) \| \leq \varepsilon / 2.
\]
  On the other hand, from the definition of $n = n_k$, we have
\[
\| A_n g_i \| \leq \| u_i \| / (2n) \leq  \varepsilon / 4
\]
and
\[
\| A_{M(n)} g_i \| \leq \| u_i \| / (2n) \leq  \varepsilon / 4,
\]
so $\| A_M(n) f - A_n f \| \leq \varepsilon$, as required. Notice that
the argument also goes through for any sequences $\la i_k \ra$ and
$\la n_k \ra$ that grow faster than the ones we have defined, that is,
satisfy (\ref{nk:eq}) and (\ref{ik:eq}) with ``$=$'' replaced by
``$\geq$.'' In sum, we have proved the following:

\begin{lemma}
  \label{almost:there:lemma} Given $T$, $f$, $\varepsilon$, and $M$,
  sequences $\la i_k \ra$ and $\la n_k \ra$ as above, and the value
  $e$ as above, there is an $n$ satisfying $1 \leq n \leq n_{e-1}$ and
  $\| A_{M(n)} f - A_n f \| \leq \varepsilon$. 
\end{lemma}

This is almost the explicit version of the ergodic theorem that we
have promised. The problem is that the bound, $i_e$, is expressed in
terms of sequence of values $\| u_i \|$ as well as the parameters $M$,
$f$, and $\varepsilon$. The fact that the term $\|T^i f - T^{i+1} f\|$
appears in the denominator of a fraction in the definition of the
sequence $\la u_i \ra$ makes it impossible to obtain an upper bound in
terms of the other parameters. But we can show that if, for any $i$,
$\|T^i f - T^{i+1} f\|$ is sufficiently small (so $T^i f$ is almost a
fixed point of $T$), we can find alternative bounds on an $n$
satisfying the conclusion of our constructive mean ergodic theorem.
Thus we can obtain the desired bounds on $n$ by reasoning by cases: if
$T^i f - T^{i+1} f$ is sufficiently small for some $i$, we are done;
otherwise, we can bound $\| u_i \|$.

The analysis is somewhat simpler in the case where $T$ is an isometry,
since then $\| T^i f - T^{i+1} f \| = \| f - T f \|$ for every
$i$. Let us deal with that case first.

\begin{lemma} 
\label{uBoundIsometry} 
If $T$ is an isometry, then for
  any $m \geq 1$ and $\varepsilon > 0$, one of the following holds:
  \begin{enumerate}
  \item $\|A_m f - f\| \leq \varepsilon$, or
  \item $\|u_i\| \leq \frac{(i+1)m \|f\|^2 }{2 \varepsilon}$ for
    every $i$.
  \end{enumerate}
\end{lemma}

\begin{proof}
  By the Cauchy-Schwartz inequality we have $\|u_0\| \leq
  \|f\|^2 / \|f - Tf\|$. By Lemma~\ref{amn:bound:lemma}, if $\|f
  - Tf\| \leq 2 \varepsilon/m$ then $\|A_m f - f\| \leq \varepsilon$.

  Otherwise, $2 \varepsilon/m < \| f - T f \| = \|T^i f - T^{i+1} f\|$
  for every $i$. In that case, we have $\|u_0\| \leq \frac{m
    \|f\|^2}{2 \varepsilon}$, and, since $\| f - g_i \| \leq \| f \|$,
  we obtain
\[
\| u_{i+1} \| \leq \| u_i \| + \frac{m \|f\| \| f - g_i \|}{2
  \varepsilon} \leq \| u_i \| + \frac{m \|f\|^2}{2 \varepsilon}
\]
for every $i$. The result follows by induction on $i$.
\end{proof}

We can now obtain the desired bounds. If $n = 1$ does not satisfy $\|
A_{M(n)} f - A_n f \| \leq \varepsilon$, we have $n_k \leq \lceil
\frac{(i_k+1) M(1) \|f\|^2}{\varepsilon^2} \rceil$ for each $k$.
Otherwise, let $K$ be any nondecreasing function satisfying $K(n) \geq
M(n) \geq n$ for every $n$. From the definition of the sequence $\la
i_k \ra$, we can extract a function $\widehat K(i)$ such that for
every $k$, $\widehat K^k(1) \geq i_k$:
\begin{itemize}
\item $\rho = \left\lceil \| f \| / \varepsilon \right\rceil$ 
\item $\widehat K(i) = i + 2^{13} \rho^4 K((i+1) K(1) \rho^2)$
\item $e = 2^9 \rho^2$
\end{itemize}
As long as $f$ is nonzero, we have $\rho \geq 1$, which ensures that
$\widehat K^e(1) \geq n_{e-1}$ and $\widehat K^e(1) \geq 1$. Thus, we
have $\| A_{M(n)} f - A_n f \| \leq \varepsilon$ for some $n \leq
\widehat K^e(1)$.

On the other hand, given a nondecreasing function $K$ to serve as a
bound for $M$, the best a counterexample function $M(n)$ can do is
return any $m$ in the interval $[n,K(n)]$ satisfying $\| A_m f - A_n f
\| > \varepsilon$, if there is one. Thus, we have
the following:

\begin{theorem}
  Let $T$ be an isometry on a Hilbert space, and let $f$ be any
  nonzero element of that space. Let $K$ be any nondecreasing function
  satisfying $K(n) \geq n$ for every $n$, and let $\widehat K$ be as
  defined above. Then for every $\varepsilon > 0$, there is an $n$
  satisfying $1 \leq n \leq \widehat K^e(1)$, such that for every $m$
  in $[n,K(n)]$, $\| A_m f - A_n f \| \leq \varepsilon$.
\end{theorem}

This is our explicit, constructive version of the mean ergodic
theorem, for the case where $T$ is an isometry. If $T$ is merely
nonexpansive instead of an isometry, the argument is more complicated
and requires a more general version of Lemma~\ref{amn:bound:lemma}.

\begin{lemma} \label{fact3Gen} Assume $T$ is a nonexpansive mapping on
  a Hilbert space, $f$ is any element, $m \geq n \geq 1$, and
  $\varepsilon > 0$. Then for any $k$, if $n \geq 2 k \| f \| /
  \varepsilon > k$, then either $\| T^k f - T^{k+1} f \| > \varepsilon /
  (2m)$ or $\| A_m f - A_n f \| \leq \varepsilon$.
\end{lemma}

\begin{proof}
We have
\begin{align*}
    \|A_m f - A_n f\| & = \frac{1}{mn} \| n
    \sum\limits^{m-1}_{i=0} T^i f - m \sum\limits^{n-1}_{j=0}
    T^j f \| \\    
    & \leq \frac{1}{mn} \| n \sum\limits^{k-1}_{i=0} T^i f - m
    \sum\limits^{k-1}_{j=0} T^j f \| \\
    & \quad\quad + \frac{1}{mn} \| n \sum\limits^{m-1}_{i=k}
    T^i f - m \sum\limits^{n-1}_{j=k} T^j f \| \\
    & \leq \frac{1}{mn} \|(n-m) \sum\limits^{k-1}_{j=0} T^j f\|\\ 
    & \quad\quad + \frac{1}{mn} \| n \sum\limits^{m-k-1}_{i=0}
    T^i (T^k f) - m \sum\limits^{n-k-1}_{j=0} T^j (T^k f) \|.
\end{align*}

The first term is less than or equal to 
\[
\frac{(m-n)}{mn} \| \sum_{i = 0}^{k-1} T^j f \| \leq \frac{k \| f
  \|}{n} \leq \varepsilon / 2.
\]
Using an argument similar to the one used in the proof of
Lemma~\ref{amn:bound:lemma}, we have
\begin{multline*}
  \frac{1}{nm} \| n \sum\limits^{m-k-1}_{i=0} T^i (T^k f) - m
  \sum\limits^{n-k-1}_{j=0} T^j (T^k f) \| \leq \\
  (m - n) \|T^k f - T^{k+1} f\| \leq m \|T^k f - T^{k+1} f\|
\end{multline*}
If $\|T^k f - T^{k+1} f\| \leq \frac{\varepsilon}{2m}$, the second
term in the last expression is also less than or equal to $\varepsilon
/ 2$, in which case $\|A_m f - A_n f\| \leq \varepsilon$.
\end{proof}

We now have an analogue to Lemma~\ref{uBoundIsometry} for the
nonexpansive case.

\begin{lemma}
\label{uNonexpansive}
For any $i \geq 0$, $n \geq 1$, and $\varepsilon > 0$, either
\begin{enumerate}
\item there is an $n \leq 2 i \lceil \frac{\| f \|}{\varepsilon}
  \rceil$ such that $\| A_{M(n)} f - A_n f \| \leq \varepsilon$, or 
\item $\| u_i \| \leq \frac{\| f \|^2}{2 \varepsilon} \sum_{j=0}^i
  M(2 j \lceil \frac{\| f \|}{\varepsilon} \rceil)$
\end{enumerate}
\end{lemma}

\begin{proof}
  Use induction on $i$. At stage $i+1$, if clause $1$ doesn't hold, we
  have $\| A_{M(i+1)} f - A_{i+1} f \| > \varepsilon$, in which case
  we can use the inductive hypothesis, the previous lemma, and the
  definition of $u_{i+1}$ to obtain clause 2.
\end{proof}

The definition of the sequences $\la i_k \ra$ and $\la n_k
\ra$ remain valid. What has changed is that we now have a more
complex expression for the bounds on $n_k$ in the case where case 2 of
Lemma~\ref{uNonexpansive} holds for each $i_k$. In other words, we
have that for every $k$,
\[
    n_k \leq \left\lceil
    \frac{\|f\|^2}{\varepsilon^2} \sum^{i_k}_{l=0} M(2l \lceil
    \|f\| / \varepsilon \rceil) \right\rceil.
\]
  unless there is an $n \leq 2 i_k \lceil \| f \| / \varepsilon
  \rceil$ such that $\| A_{M(n)} f - A_n f \| \leq
  \varepsilon$. Assuming $K$ is a nondecreasing function satisfying
  $K(n) \geq M(n)$, we can replace this last bound by $\left\lceil
    \frac{\|f\|^2}{\varepsilon^2} (i_k + 1) K(2 i_k \lceil
    \|f\| / \varepsilon \rceil) \right\rceil$. Define
\begin{itemize}
\item $\rho = \lceil \| f \| / \varepsilon \rceil$
\item $\overline K(i) = i + 2^{13} \rho^4 K((i+1) K(2 i \rho) \rho^2)$
\item $e = 2^9 \rho^2$
\end{itemize}
Then we have:

\begin{theorem}
\label{constructive:met:final}  
Let $T$ be an nonexpansive linear operator on a Hilbert space, and let
$f$ be any nonzero element of that space. Let $K$ be any nondecreasing
function satisfying $K(n) \geq n$ for every $n$, and let $\overline K$
be as defined above. Then for every $\varepsilon > 0$, there is an $n$
satisfying $1 \leq n \leq \overline K^e(1)$, such that for every $m$
in $[n,K(n)]$, $\| A_m f - A_n f \| \leq \varepsilon$.
\end{theorem}

Direct calculation yields the following asymptotic bounds.

\begin{theorem}
  Let $T$ be any nonexpansive map on a Hilbert space, let $K$ be any
  nondecreasing function satisfying $K(n) \geq n$ for every $n$, and
  for every nonzero $f$ and $\varepsilon > 0$, let $r_K(f,
  \varepsilon)$ be the least $n \geq 1$ such that $\| A_m f - A_n f \|
  \leq \varepsilon$ for every $m$ in $[n,K(n)]$.
\begin{itemize}
\item If $K = n^{O(1)}$, then $r_K(f, \varepsilon) = 2^{2^{O(\rho^2
      \log \log \rho)}}$.
\item If $K = 2^{O(n)}$, then $r_K(f, \varepsilon) = 2^1_{O(\rho^2)}$.
\item If $K = O(n)$ and $T$ is an isometry, then $r_K(f, \varepsilon)
  = 2^{O(\rho^2 \log \rho)}$.
\end{itemize}
In these expressions, $\rho$ abbreviates $\left\lceil \| f \| /
  \varepsilon\right\rceil$ and $2_n^x$ denotes the $n$th iterate of
$y \mapsto 2^y$ starting with $x$.
\end{theorem}

Alternatively, we can fix $\rho$ and consider the dependence on
$K$. Here are two special cases.

\begin{theorem}
Let $T$ be an isometry on a Hilbert space, and let $K$ be as
above. Fix $\rho = \left\lceil \| f \| / \varepsilon \right\rceil$.
\begin{itemize}
\item If $K(x) = x + c$, then, as a function of $c$, $r_K(f,
  \varepsilon) = O(c)$.
\item If $K(x) = c x + d$, then, as a function of $c$,
  $r_K(f,\varepsilon) = c^{O(1)}$.
\end{itemize}
\end{theorem}


\section{A quantitative pointwise ergodic theorem}
\label{pointwise:section}

Let $\tau$ be a measure preserving transformation on a probability
space $\la \mdl X, \mdl B, \mu \ra$, and let $T$ be the Koopman
operator on the space $L^1(\mdl X)$, defined by $T f = f \circ
\tau$. The mean ergodic theorem implies that for any $f$ in $L^2(\mdl
X)$, the ergodic averages converge in the $L^2$ norm. But since, for
any $f$ in $L^2(\mdl X)$, $\| f \|_1 \leq \| f \|_2$, this implies
convergence in the $L^1$ norm also. We also have that $L^2(\mdl X)$ is
dense in $L^1(\mdl X)$, and if $\| f - f' \|_1 < \varepsilon$ then $\|
A_n f - A_n f' \|_1 < \varepsilon$ for every $n$. As a result, we have
convergence for every $f$ in $L^1(\mdl X)$ as well.

Birkhoff's pointwise ergodic theorem makes a stronger assertion:
\begin{theorem}
  Let $T$ be the Koopman operator corresponding to a measure
  preserving transformation $\tau$ on a probability space $\mdl X$,
  and let $f$ be any element of $L^1(\mdl X)$. Then $\la A_n f \ra$
  converges pointwise, almost everywhere.
\end{theorem}
This is equivalent to the following:
\begin{theorem}
\label{explicit:pet}
  Given $T$ and $f$ as above, for every $\lambda_1 > 0$ and $\lambda_2
  > 0$ there is an $n$ such that for every $k \geq n$,
\[
\mu(\{ x \st \max_{n \leq m \leq k} | A_m f (x) - A_n f (x) | >
\lambda_1 \}) \leq \lambda_2.
\]
\end{theorem}
By the logical manipulations described in the introduction, this, in
turn, is equivalent to the following:
\begin{theorem}
\label{constructive:pet}
Given $T$ and $f$ as above, for every $\lambda_1 > 0$, $\lambda_2 >
0$, and $K$ there is an $n \geq 1$ satisfying
\[
\mu(\{ x \st \max_{n \leq m \leq K(n)} | A_nf(x) - A_m f(x) | >
\lambda_1 \}) \leq \lambda_2.
\]
\end{theorem}
Using the maximal ergodic theorem, described below, one can reduce all
three versions of the pointwise ergodic theorem to the case where $f
\in L^2(\mdl X)$. In this section, we will provide a constructive
proof of Theorem~\ref{constructive:pet} for this restricted case, with
an explicit bound on $n$, expressed in terms of $K$, $\| f \|_2$,
$\lambda_1$, and $\lambda_2$.

The maximal ergodic theorem can be stated as follows:
\begin{theorem}
\label{maximal:thm}
  Suppose $n \geq 1$, and let $A = \{ x \st \max_{i \leq n} \sum_{j < i} T^j
  f (x) > 0 \}$. Then $\int_A f d\mu \geq 0$.  
\end{theorem}
The proof of this is essentially constructive (see \cite{spitters:06},
and the proofs in \cite{billingsley:78,walters:00}). We will make use
of the following corollary:
\begin{corollary}
\label{maximal:cor}
For any $\lambda > 0$ and $n \geq 1$, 
\[
\mu(\{ x \st \max_{1 \leq i \leq n} |A_i f (x)| > \lambda\} ) \leq \|
f \|_1 / \lambda.
\] 
\end{corollary}

\begin{proof}
We have
\begin{align*}
\{ x \st \max_{1 \leq i \leq n} |A_i f (x)| > \lambda\} & \subseteq 
\{ x \st \max_{1 \leq i \leq n} A_i (|f|) (x) > \lambda\} \\
& = \{ x \st \max_{1 \leq i \leq n} A_i (|f| - \lambda)(x) > 0 \} \\
& = \{ x \st \max_{1 \leq i \leq n} \sum_{j < i}T^j(|f| - \lambda)(x) > 0 \}.
\end{align*}
Call this last set $A$. By Theorem~\ref{maximal:thm}, we have 
$\int_A (|f| - \lambda) \; d\mu \geq 0$, and hence
\[
\| f \|_1 = \int |f| d\mu \geq \int_A | f | d\mu \geq \lambda \mu(A),
\]
so $\mu(A) \leq \| f \|_1 / \lambda$.
\end{proof}
Taking limits, Corollary~\ref{maximal:cor} implies that for every
$\lambda > 0$, 
\[
\mu(\{ x \st \sup_{i \geq 1} |A_i f (x)| > \lambda\} ) \leq \| f
\|_1 / \lambda. 
\]
We will stick with the formulation above, however, to emphasize the
combinatorial character of our arguments.

Most contemporary presentations of the pointwise ergodic theorem
proceed to define $f^*(x) = \limsup A_nf(x)$ and $f_*(x) = \liminf
A_nf(x)$, and then use the maximal ergodic theorem to show that the
two are equal almost everywhere. Billingsley \cite{billingsley:78},
however, presents a proof that makes use of the $L^2$ limit of $A_n
f(x)$, as guaranteed to exist by the mean ergodic theorem, rather
than $f^*$ and $f_*$. We will ``mine'' this proof for our constructive
version.

The idea is as follows. For the moment, let $h$ denote the $L^2$ limit
of $\la A_n f \ra$, and for each $i$ let $f_i = h + g_i$, where the
sequence $\la g_i \ra$ is as defined in the last section. Then $f =
\lim_i f_i$. For any $m \geq 1$, $n \geq 1$, $i \geq 0$, and $x \in
\mdl X$, we have
\begin{equation}
\label{pointwise:calc}
\begin{split}
  | A_m f(x) - A_n f(x) | \leq & | A_m(f - f_i)(x)| + |A_m f_i (x) -
  A_n f_i(x)| + \\ & \quad \quad |A_n(f - f_i)(x) | \\
  & \leq | A_m(f - f_i)(x)| + |A_m h(x) - A_n h(x) | + \\
  & \quad \quad |A_m g_i(x) | + | A_n g_i(x) |
  + |A_n (f - f_i)(x) |.
\end{split}
\end{equation}
Using the maximal ergodic theorem, the first and last terms can be
made small outside a small set of exceptions, for all values of $m$
and $n$ simultaneously, by taking $i$ big enough so that $\| f - f_i
\|_2$ is small. The second term is equal to $0$, since $h$ is a fixed
point of $T$. Finally, using the fact that $g_i$ is of the form $u -
Tu$, the third and fourth terms can be made
arbitrarily small in the $L^2$ norm, and hence small outside a small
set of exceptions, by taking $m$ and $n$ sufficiently large.

The problem is that in our constructive version we do not have access
to $h$, which is equal to $f - \lim_i g_i$; nor can we determine how
large $i$ has to be to make $\| f - f_i \|_2$ sufficiently
small. Instead, we replace $h$ by an approximation $f - g_j$.  Then we
get
\begin{equation*}
\begin{split}
  | A_m f(x) - A_n f(x) | & \leq | A_m(g_j - g_i)(x)| + \\
  & \quad \quad |A_m (f -
  g_j + g_i)(x) -
  A_n (f - g_j + g_i)(x)| + \\
  & \quad \quad |A_n(g_j - g_i)(x) | \\
  & \leq | A_m(g_j - g_i)(x)| + |A_m (f - g_j)(x) - A_n (f - g_j)(x)
  | + \\
  & \quad \quad |A_m g_i(x) | + | A_n g_i(x) |
  + |A_n (g_j - g_i)(x) |.
\end{split}
\end{equation*}
Similar considerations now hold: we can make the first and last terms
small outside a small set of exceptions, independent of $m$ and $n$,
by ensuring that $\| g_j - g_i\|_2$ is sufficiently small. We can make the
second term small using Lemma~\ref{pet:key:lemma}, with an appropriate
choice of $j$. Finally, the third and fourth terms are small outside
a small set of exceptions when $m$ and $n$ are sufficiently large.

Thus our task is to find a value of $n$ such that
\begin{multline}
\label{main:pet:equation}
\max_{n \leq m \leq K(n)} ( | A_m(g_j - g_i)(x)| + |A_m (f -
  g_j)(x) - A_n (f - g_j)(x) | + \\ |A_m g_i(x) | + | A_n g_i(x) |
  + |A_n (g_l - g_i)(x) | )
\end{multline}
is less than or equal to $\lambda_1$, outside a set of size at most
$\lambda_2$. We will consider the various components of this sum, in
turn.

We will make use of Chebyshev's inequality, which shows that $| f(x)
|$ is small, outside a small set of exceptions, when $\| f \|_2$ is
small:

\begin{lemma}
\label{l2bound}
For any $\lambda \geq 0$, $\mu(\{ x \st | f(x) | \geq \lambda \}) \leq \|
f \|^2_2 / \lambda^2$.
\end{lemma}

\begin{proof}
Otherwise, we would have $\| f \|^2_2 = \int | f |^2 d\mu > \lambda^2
( \|f \|^2_2 / \lambda^2) = \| f \|^2_2$.
\end{proof}

The next lemma deals with the first and last terms in
(\ref{main:pet:equation}). 

\begin{lemma}
\label{first:last:terms}
Suppose $\| g_j - g_i \|_2 \leq \lambda_1 \lambda_2 / 8$. Then for
any $k$ and $n$ satisfying $1 \leq n \leq k$, we have
\[
\mu(\{ x \st \max_{n \leq m \leq k} (|A_m(g_j - g_i)(x)| + |A_n(g_j
- g_i)(x)|) > \lambda_1/2 \} ) \leq \lambda_2 / 2.
\]
\end{lemma}

\begin{proof}
By Corollary~\ref{maximal:cor}, we have
\[
\mu(\{ x \st \max_{1 \leq m \leq k} |A_m (g_j - g_i) (x)| > \lambda_1 / 4 \}
) \leq (\lambda_1 \lambda_2 / 8) / (\lambda_1 / 4) = \lambda_2 / 2.
\]
Since $| A_m(g_j - g_i)(x)| + |A_n(g_j - g_i)(x)| > \lambda_1/2$
implies that either $| A_m(g_j - g_i)(x)| > \lambda_1 / 4$ or $|
A_n(g_j - g_i)(x)| > \lambda_1 / 4$, the set in the last displayed
formula includes the set in the statement of the lemma.
\end{proof}

The next four lemmas concern the third and fourth terms of
(\ref{main:pet:equation}), which are of the form $A_n g = A_n (u - Tu)
= (u - T^n u) / n$. To show that we can make these terms small outside
a small set by making $n$ sufficiently large, we will need to split
$u$ into two components, one of which is bounded, and the other of
which is small in the $L^1$ norm. The following lemma enables us to do
this.

\begin{lemma}
\label{u:decomp:lemma}
For any $u \in L^2(\mdl X)$ and $L > 0$, write $u = u' + u''$, where
\[
u'(x) = \left\{
\begin{array}{ll}
u(x) & \mbox{if $|u(x)| \leq L$} \\
0 & \mbox{otherwise}
\end{array}
\right.
\]
and $u'' = u - u'$. Then $\| u' \|_{\infty} \leq L$, and $\| u''\|_1
\leq \| u \|^2_2 / L$.
\end{lemma}

\begin{proof}
The first claim is immediate. For the second, we have
\[
\| u''\|_1 = \int_{\{ x \st | u(x) | \geq L \}} |u(x)| \; d\mu \leq
\int_{\{ x \st | u(x) | \geq L \}} u^2(x)/L \; d\mu \leq \| u \|^2_2 / L,
\]
as required.
\end{proof}

\begin{lemma}
\label{other:terms:lemma}
  Let $u \in L^2(\mdl X)$, let $g = u - Tu$, and suppose $n \geq
  \frac{2^{12} \| u \|_2^2}{\lambda_1^2 \lambda_2}$ and $n \geq
  1$. Then for any $k \geq n$,
\[
\mu( \{ x \st \max_{n \leq m \leq k} ( |A_m g(x)| + |A_n g(x)|)
> \lambda_1/4 \} ) \leq \lambda_2 / 4.
\]
\end{lemma}

\begin{proof}
Let $L = 2^7 \| u \|^2_2 / \lambda_1 \lambda_2$, and let $u = u' +
u''$ be the decomposition in Lemma~\ref{u:decomp:lemma}. Then 
\begin{multline}
\label{u:bound:equation}
|A_m g_i(x)| + |A_n g_i(x)| \leq
|A_m (u' - Tu')(x)| + |A_n (u' - Tu')(x)| + \\
|A_m (u'' - Tu'')(x)| + |A_n (u'' - Tu'')(x)|.
\end{multline}
Then we have $\| u'' - Tu'' \|_1 \leq 2 \| u''\|_1 \leq 2 \| u \|^2_2 / L$,
and so, by Corollary~\ref{maximal:cor},
\[
\mu ( \{ x \st \max_{1 \leq n \leq k} | A_n(u'' - Tu'')(x) | >
\lambda_1 / 16 \}) \leq \frac{32 \| u \|^2_2}{L \lambda_1} = \lambda_2 /
4.
\]
So for $1 \leq n \leq m \leq k$, the sum of the last two terms on the
right-hand side of (\ref{u:bound:equation}) is at most $\lambda_1 / 8$
outside a set of measure $\lambda_2 / 4$. On the other hand, for every
$x$ and $m \geq n$, $|A_m (u' - Tu')(x)|$ and $|A_n (u' - Tu')(x)|$
are bounded by $2 \|u'\|_\infty / n \leq 2 L /n \leq \lambda_1 / 16$,
since $n \geq 32 L / \lambda_1$.
\end{proof} 

Taking $k = K(1)$ in the following lemma shows that if $n = 1$ does
not provide a witness to the conclusion of our constructive pointwise
ergodic theorem, we can bound the terms $\| u_i \|_2$. This is
analogous to Lemma~\ref{uBoundIsometry}.

\begin{lemma} 
\label{uBoundPET}
For any $k \geq 1$, one of the following holds:
\begin{enumerate}
\item $\mu(\{ x \st \max_{1 \leq m \leq k} | (A_m f - f)(x)| >
  \lambda_1 \}) \leq \lambda_2$, or
\item $\| u_i \|_2 \leq \frac{(i+1) \| f \|^2_2 k^{3/2}}{\lambda_1
    \sqrt{\lambda_2}}$ for every $i$.
\end{enumerate}
\end{lemma}

\begin{proof}
  Suppose, first, $\| T f - f \|_2 \leq 2 \lambda_1 \sqrt{\lambda_2} /
  k^{3/2}$. Then by Lemma~\ref{amn:bound:lemma}, for every $m$
  satisfying $1 \leq m \leq k$ we have $\| A_m f - f \|_2 \leq \lambda_1
  \sqrt{\lambda_2} / \sqrt{k}$. By Lemma~\ref{l2bound}, this implies
\[
\mu(\{ x \st |(A_m f - f)(x)| > \lambda_1 \}) \leq (\lambda_1^2
\lambda_2 / k) / \lambda_1^2 = \lambda_2/k
\]
for each $m$ in the interval $[1, k]$. So, in that case, clause 1 holds.

Otherwise, we have $\| T f - f \|_2 > 2 \lambda_1 \sqrt{\lambda_2} /
k^{3/2}$. In that case, clause 2 follows from the definition
of the sequence $\la u_i \ra$, as in the proof of
Lemma~\ref{uBoundIsometry}.
\end{proof}

Combining the last two lemmas, we have the following.

\begin{lemma}
\label{final:ui:lemma}
Either $\mu(\{ x \st \max_{1 \leq m \leq K(1)} | (A_m f - f)(x)| >
\lambda_1 \}) \leq \lambda_2$, 
or, for every $i$, if $n \geq \frac{2^{12} K(1)^3 i^2 \| f
  \|_2^4}{\lambda_1^4 \lambda_2^2}$, then for any $k \geq n$ we have
\[
\mu( \{ x \st \max_{n \leq m \leq k} ( |A_m g_i(x)| + |A_n g_i(x)|)
> \lambda_1/4 \} ) \leq \lambda_2 / 4.
\]
\end{lemma}

Finally, we address the second term of (\ref{main:pet:equation}),
using Lemma~\ref{pet:key:lemma}.

\begin{lemma}
\label{second:term}
  Given $1 \leq n \leq k$, let $e = e(k,\lambda_1,\lambda_2) = 2^{34}
  k^7 \lceil \| f \|_2 / (\lambda_1 \sqrt {\lambda_2})\rceil^6$. Then
  for any $i$, there is a $j$ in the interval $[i,i+e)$ satisfying
\[
  \mu(\{ x \st \max_{n \leq m \leq k} | A_m(f - g_j)(x) - A_n (f -
  g_j)(x)| > \lambda_1 / 4 \} ) \leq \lambda_2 / 4.
\]
\end{lemma}

\begin{proof}
Working backwards, it is enough to ensure that for every $m$
satisfying $n \leq m \leq k$, 
\[
  \mu(\{ x \st | A_m(f - g_j)(x) - A_n (f -
  g_j)(x)| > \lambda_1 / 4 \} ) \leq \frac{\lambda_2}{4k}.
\]
By Lemma~\ref{l2bound}, it suffices to ensure that
\[
\frac{16 \| A_m(f - g_j) - A_n (f - g_j) \|^2_2}{\lambda_1^2} \leq
\frac{\lambda_2}{4k}, 
\]
for each such $m$, i.e.
\[
\| A_m(f - g_j) - A_n (f - g_j) \|_2 \leq \frac{\lambda_1
  \sqrt{\lambda_2}} {8 k}.
\]
Substituting the right hand side for $\varepsilon$ in
Lemma~\ref{pet:key:lemma}, this is guaranteed to happen for some $j$
in $[i,i+e)$.
\end{proof}

We can finally put all the pieces together. Set $i_0 = 0$.
Assuming $i_k$ has been defined, let
\[
n_k = \left\lceil \frac{2^{12} K(1)^3 i_k^2 \| f \|_2^4}{\lambda_1^4
  \lambda_2^2} \right\rceil
\]
and
\[
i_{k+1} = i_k + e(K(n_k),\lambda_1,\lambda_2) = i_k + 2^{34}
  K(n_k)^7 \lceil \| f \|_2 / (\lambda_1 \sqrt {\lambda_2})\rceil^6.
\]
Finally, define
\[
e = \left\lceil \frac{2^7 \| f \|^2_2}{\lambda_1 \sqrt{\lambda_2}}
\right\rceil 
\]
Then we have:
\begin{lemma}
  Let $\lambda_1 > 0$, let $\lambda_2 > 0$, and let $K$ be any
  function. Given $f$, define the sequence $\la i_k \ra$ and the value
  $e$ as above. Then there is an $n$ satisfying $1 \leq n \leq n_e$
  and
\[
\mu(\{ x \st \max_{n \leq m \leq K(n)} | A_m f(x) - A_n f(x) | >
\lambda_1 \}) \leq \lambda_2.
\]
\end{lemma}

\begin{proof}
  Suppose $n = 1$ does not witness the conclusion. As in the proof of
  Lemma~\ref{almost:there:lemma}, there is some $k < e$ such that for
  $| a_{i_{k+1}} - a_{i_k} | \leq \frac{\lambda_1^2 \lambda_2^2}{2^7 \| f
    \|_2}$. Set $n = n_k$ and $i = i_k$. By
  Lemma~\ref{second:term}, there is a $j$ in the interval
  $[i,i_{k+1})$ satisfying
\[
  \mu(\{ x \st \max_{n \leq m \leq K(n)} | A_m(f - g_j)(x) - A_n (f -
  g_j)(x)| > \lambda_1 / 4 \} ) \leq \lambda_2 / 4.
\]
Since $i \leq j \leq i_{k+1}$ we have $| a_j - a_i | \leq |a_{i_{k+1}}
- a_j|$ and so $\| g_i - g_j \| \leq \lambda_1 \lambda_2 / 8$.
By Lemma~\ref{first:last:terms}, we have
\[
\mu(\{ x \st \max_{n \leq m \leq K(n)} (|A_m(g_j - g_i)(x)| + |A_n(g_j
- g_i)(x)|) > \lambda_1/2 \} ) \leq \lambda_2 / 2.
\]
  By Lemma~\ref{final:ui:lemma} we have
\[
\mu( \{ x \st \max_{n \leq m \leq K(n)} ( |(A_m g_i)(x)| + |(A_n g_i)(x)|)
> \lambda_1/4 \} ) \leq \lambda_2 / 4.
\]
  The result follows.
\end{proof}

Once again, we can do some final housecleaning.  Define:
\begin{itemize}
\item $\rho = \lceil \| f \|/ (\lambda_1 \sqrt {\lambda_2}) \rceil$ 
\item $\widehat K(i) = i + 2^{34} \rho^6 K(2^{12} K(1)^3 i^2 \rho^4)$
\item $e = \lceil 2^7 \| f \|^2_2 / (\lambda_1 \sqrt{\lambda_2})
  \rceil$
\end{itemize}

\begin{theorem}
\label{final:pet:theorem}
Let $\tau$ be any measure preserving transformation of a finite
measure space $\mdl X$, and let $f$ be any nonzero element of
$L^2(\mdl X)$. Let $\lambda_1 > 0$, let $\lambda_2 > 0$, and let $K$
be any function. Then, with the definitions above, there is an $n$
satisfying $1 \leq n \leq \overline K^e(1)$ and
\[
\mu(\{ x \st \max_{n \leq m \leq K(n)} | A_m f(x) - A_n f(x) | >
\lambda_1 \}) \leq \lambda_2.
\]
\end{theorem}


\section{Results from upcrossing inequalities}
\label{upcrossing:section}

We are not the first to develop constructive versions of the ergodic
theorems. Let $\tau$ be a measure preserving transformation on a
finite measure space $\mdl X = \la X, \mdl B, \mu \ra$, and for every
real $\alpha < \beta$, let $\omega_{\alpha,\beta}(x)$ denote the
number of times the sequence $\la A_n f (x) \ra_{n \in \NN}$ upcrosses
the interval $[\alpha,\beta]$, that is, proceeds from a value $A_i
f(x)$ less than $\alpha$ to a value $A_j f(x)$ greater than $\beta$.
For any $x \in X$, the statement that $\la A_n f (x) \ra_{n \in \NN}$
converges is clearly equivalent to the statement that, for every
$\alpha < \beta$, $\omega_{\alpha,\beta}(x)$ is finite.
Bishop~\cite{bishop:66,bishop:67,bishop:67b} showed that for any $f$
in $L^1(\mdl X)$, we have
\[
\int_X \omega_{\alpha,\beta} \; d\mu \leq \frac{1}{\beta - \alpha} \int_X 
(f - \alpha)^+ \; d\mu.
\]
In particular, let $\Omega_{\alpha,\beta}^{k \uparrow}$ denote the set
$\{ x \st \omega_{\alpha,\beta}(x) \geq k \}$, that is, the set of
points for which the sequence makes no less than $k$
upcrossings. Bishop's inequality immediately implies
\[
\mu(\Omega_{\alpha,\beta}^{k \uparrow}) \leq \frac{1}{k}
\frac{1}{\beta - \alpha} \int_X (f - \alpha)^+ \; d\mu,
\]
from which the ordinary pointwise ergodic theorem follows in a
straightforward way. Let the set $\Omega_{[\alpha,\beta]}^{k
  \downarrow}$ be defined, analogously, to be the set of elements $x$
for which the sequence $\la A_n f(x) \ra_{n \in \NN}$ makes no less than $k$
downcrossings from a value above $\beta$ to a value below
$\alpha$. Recently, Ivanov \cite{ivanov:96} has shown that for
nonnegative functions $f$, the size of this set decays exponentially
with $k$:
\[
\mu(\Omega_{\alpha,\beta}^{k \downarrow}) \leq \left(
  \frac{\alpha}{\beta} \right)^k.
\]
A similar result was obtained, independently, by Kalikow and Weiss
\cite{kalikow:weiss:99}. Bishop's and Ivanov's results and their
consequences are explored thoroughly in \cite{kachurovskii:96}. There
has recently been a surge of interest in such upcrossing inequalities;
see, for example, \cite{hochman:unp,jones:et:al:98,jones:et:al:99}.

Upcrossing inequalities can be used in a crude way to obtain bounds on
our constructive pointwise ergodic theorem,
Theorem~\ref{constructive:pet}. They can also be used, indirectly, to
obtain bounds on our constructive mean ergodic theorem,
Theorem~\ref{constructive:met}, in the specific case where the
operator in question is the Koopman operator corresponding to a
measure preserving transformation. Of course, the upcrossing
inequalities characterize the overall oscillatory behavior of a
sequence, and thus provide a more information. On the other hand, our
results in Section~\ref{constructive:met:section} apply to any
nonexpansive mapping on a Hilbert space, and so are more general.
There are further differences: because we obtain our pointwise results
from our constructive version of the mean ergodic theorem, the $L^2$
norm $\| f \|_2$ of $f$ plays a central role. In contrast, results
obtained using upcrossing techniques are more naturally expressed in
terms of $\| f \|_1$ and $\| f \|_\infty$. In this section, we will
see that when the two methods yield analogous results, they provide
qualitatively different bounds.

First, let us show how Bishop's inequality leads to a bound on the
witness to our constructive pointwise ergodic theorem,
Theorem~\ref{constructive:pet}, when $f$ is in $L^\infty(\mdl
X)$. Note that for any $\beta > \alpha > 0$, Bishop's result implies
that the number of upcrossings of the interval $[\alpha,\beta]$
satisfies
\[
\int_X \omega_{\alpha,\beta} \; d\mu \leq \| f \|_\infty / (\beta - \alpha).
\]
By symmetry, the same bound holds for downcrossings. Given
$\lambda_1$, divide the essential range $[-\| f \|, \| f \| ]$ of $f$
into $\lceil 4 \| f \|_\infty / \lambda_1 \rceil$ intervals of size
$\lambda_1 / 2$ each. Given $K$ as in Theorem~\ref{constructive:pet},
consider the sequence of $e+1$ intervals 
\[ 
[1,K(1)], \;\; [K(1),K(K(1))], \;\; \ldots, \;\; [K^e(1), K^{e+1}(1)]. 
\]
For every $i \in [0,e]$ and $j \in [1, \lceil 4 \| f \|_\infty / \lambda_1
\rceil]$, let $A_{i,j}$ denote the set of $x$ such that $\la A_nf (x)
\ra$ upcrosses or downcrosses the $j$th interval in the essential
range of $f$ somewhere in the interval $[K^i(1), K^{i+1}(1)]$. Now,
suppose that for each $i \leq e$,
\begin{equation}
\label{bishop:bound}
\mu(\{ x \st \max_{K^i(1) \leq m \leq K^{i+1}(1)} | A_m f(x) -
A_{K^i(1)} f(x) | > \lambda_1 \}) > \lambda_2.
\end{equation}
Since each $x$ in this set either upcrosses or downcrosses one of the
$\lceil 4 \| f \|_\infty / \lambda_1 \rceil$ intervals in the range of
$f$, we have, for each $i \leq e$, 
\[
\sum_{j=1\ldots \lceil 4 \| f
  \|_\infty / \lambda_1 \rceil} \mu(A_{i,j}) > \lambda_2.
\] 
So
\[
\sum_{i=0\ldots e,j=1\ldots \lceil 4 \| f
  \|_\infty / \lambda_1 \rceil} \mu(A_{i,j}) > (e+1) \lambda_2,
\]
which means that for some $j$, $\sum_{i=0\ldots e} \mu(A_{i,j}) > (e +
1) \lambda_1 \lambda_2 / 4 \| f \|_\infty$. Let $[ \alpha, \beta ]$ be
the corresponding interval, and let $\omega'(x)$ be the number of
times $\la A_nf(x) \ra$ upcrosses or downcrosses this interval. Then
we have
\[
\sum_{i=0\ldots e} \mu(A_{i,j}) = \int \sum_{i=0\ldots e} \chi_{A_{i,j}} \; d\mu \leq
\int \omega'(x) \; d \mu \leq 2 \| f \|_\infty / (\beta - \alpha) = 4
\| f \|_\infty  / \lambda_1.
\]
In other words, we have shown that if (\ref{bishop:bound}) holds for
each $i \leq e$, we have 
\[
(e + 1) \lambda_1 \lambda_2 / 4 \| f \|_\infty \leq 4 \| f \|_\infty /
\lambda_1,
\]
which implies $e + 1 \leq 16 \| f \|_\infty^2 / \lambda_1^2
\lambda_2$. Taking the contrapositive of this claim, we have the
following analogue of Theorem~\ref{final:pet:theorem}:
\begin{theorem}
\label{alt:pet:theorem}
Let $T$ be a Koopman operator corresponding to a measure preserving
transformation of a space $\mdl X$ and let $f$ be any element of
$L^\infty(\mdl X)$. Let $\lambda_1 > 0$, let $\lambda_2 > 0$, and let
$K$ be any function satisfying $K(n) \geq n$ for every $n$. Let $e =
\lceil 16 \| f \|_\infty^2 / \lambda_1^2 \lambda_2 \rceil$. Then there
is an $n$ satisfying $1 \leq n \leq K^e(1)$ and
\[
\mu(\{ x \st \max_{n \leq m \leq K(n)} | A_m f(x) - A_n f(x) | >
\lambda_1 \}) \leq \lambda_2.
\]
\end{theorem}
In other words, we can bound a witness to our constructive pointwise
ergodic theorem by $e = \lceil 16 \| f \|_\infty^2 / \lambda_1^2
\lambda_2 \rceil$ iterations of $K$ on $1$. Compare this to our
Theorem~\ref{final:pet:theorem}, which requires asymptotically fewer
($\lceil 2^7 \| f \|_2^2 / \lambda_1 \sqrt{\lambda_2} \rceil$) iterations of
a function, $\overline K$, which, however, grows faster than $K$.

Ivanov's inequality does not seem to enable us to improve the bound in
the previous theorem. But a consequence of Ivanov's inequality,
obtained by Kachurovskii, enables us to treat the mean ergodic theorem
in a similar way. A sequence of real numbers
$a_n$ is said to admit $k$
$\varepsilon$-fluctuations if there is a sequence
\[
m_1 < n_1 \leq m_2 < n_2 \leq \ldots \leq m_k < n_k
\]
such that for every $i$, $1 \leq i \leq k$, $| a_{m_i} - a_{n_i} |
\geq \varepsilon$. Let $T$ be
the Koopman operator arising from a measure preserving transformation
on $\mdl X$. By the mean ergodic theorem, for every $\varepsilon > 0$,
the number $k_\varepsilon$ of $\varepsilon$-fluctuations is finite.
Kachurovskii  \cite[Theorem 29]{kachurovskii:96} shows:
\begin{theorem}
  Let $f$ be any element of $L^\infty(\mdl X)$. Then for every
  $\varepsilon > 0$,
\[
k_\varepsilon \leq C\left( \frac{\| f \|_\infty}{\varepsilon}
\right)^4 \left( 1 + \ln \left( \frac{\| f \|_\infty}{\varepsilon}
  \right) \right)
\]
for some constant $C$. 
\end{theorem}
Now, given any counterexample function $M$ satisfying $M(n) \geq n$
for every $n$, consider the sequence 
\[
A_1 f, A_{M(1)}f, A_{M(M(1))}f, \ldots, A_{M^{k_\varepsilon + 1}(1)}
f.
\] 
At least one step must change by less than $\varepsilon$. Thus, we
have the following analogue to our
Theorem~\ref{constructive:met:final}:
\begin{theorem}
  Let $T$ be a Koopman operator corresponding to a measure preserving
  transformation of a space $\mdl X$ and let $f$ be any element of
  $L^\infty(\mdl X)$. Let $K$ be any function satisfying
  $K(n) \geq n$ for every $n$. Let $k(f,\varepsilon)$ be the bound on
  $k_{\varepsilon}$ given in the preceding theorem. Then for every
  $\varepsilon > 0$, there is an $n$, $1 \leq n \leq
  K^{k(f,\varepsilon)}(1)$, satisfying $\| A_m f - A_n f \| \leq
  \varepsilon$ for every $m \in [n, K(n)]$.
\end{theorem}
In other words, we can bound a witness to the conclusion of the
constructive mean ergodic theorem with $k(f,\varepsilon)$ iterates of
$K$. In contrast, Theorem~\ref{constructive:met:final} required
$e(f,\varepsilon) = C \lceil \| f \|_2 / \varepsilon \rceil$ iterates
of a faster-growing function $\overline K$. 


\section{Computability of rates of convergence}
\label{computability:section}

Suppose $\la a_n \ra_{n \in \NN}$ is any sequence of rational
numbers that decreases monotonically to $0$. No matter how slowly the
sequence converges, if one is allowed to query the values of the
sequence, one can compute a function $r(\varepsilon)$ with the
property that for every rational $\varepsilon > 0$ and every $m >
r(\varepsilon)$, $| a_m - a_{r(\varepsilon)} | < \varepsilon$. The
algorithm is simple: on input $\varepsilon$, just search for an $m$
such that $a_m < \varepsilon$.

On the other hand, it is not hard to construct a computable sequence
$\la a_n \ra_{n \in \NN}$ of rational numbers that converges to $0$,
with the property that no computable function $r(\varepsilon)$ meets
the specification above. This is an easy consequence of the
unsolvability of the halting problem. Let $\la M_i \ra_{i > 0}$ be an
enumeration of Turing machines, and let $j_i$ be an enumeration of the
natural numbers with the property that every natural number appears
infinitely often in the enumeration. For every $i$, let $a_i = 1 /
j_i$ if Turing machine $M_{j_i}$, when started with input $0$, halts
in less than $i$ steps, but not in $i'$ steps for any $i' < i$ such
that $j_{i'} = j_i$; and let $a_i = 0$ otherwise. Then $\la a_j \ra$
converges to $0$, since once we have recognized all the machines among
$M_1, \ldots, M_n$ that eventually halt, $a_i$ remains below $1 / n$.
But any any value $r(1 / n)$ meeting the specification above tells us
how long we have to wait to determine whether $M_n$ halts, and so any
such such $r$ would enable us to solve the halting problem.

In a similar way, one can construct a computable sequence
$\la a_n \ra_{n \in \NN}$ of rational numbers that is monotone and
bounded, but converges to a noncomputable real number. This, too,
implies that no computable function $r(\varepsilon)$
meets the specification above. Such a sequence is known as a
\emph{Specker sequence}, and an example is given in the proof of
Theorem~\ref{noncomputability:thm}, below. Thus neither monotonicity
nor the existence of a computable limit alone is enough to guarantee
the effective convergence of a sequence of rationals.

What these examples show is that the question as to whether it is
possible to compute a bound on a rate of convergence of a sequence
from some initial data is not a question about the speed of the
sequence's convergence, but, rather, its predictability. In this
section, we show that in general, one cannot compute a bound on the
rate of convergence of ergodic averages from the initial data,
although one can do so when dealing when dealing with an ergodic
transformation of a (finite) measure space.

The results in this section presuppose notions of computability for
various objects of analysis. There are a number of natural, and
equivalent, frameworks for defining such notions. For complete detail,
the reader should consult Pour el and Richards
\cite{pourel:richards:89} or Weihrauch \cite{weihrauch:00}. To make
sense of the results below, however, the following sketchy overview
should suffice.

The general strategy is to focus on infinitary objects that can be
represented with a countable set of data. For example, a real number
can be taken to be represented by a sequence of rational numbers
together with a bound on its rate of convergence; the corresponding
real number is said to be computable if it has a computable
representation. In other words, a computable real number is given by
computable functions $a : \NN \to \QQ$ and $r: \QQ \to \NN$ with the
property that for every rational $\varepsilon > 0$, $| a_m -
a_{r(\varepsilon)}| < \varepsilon$ for every $m \geq r(\varepsilon)$.
A function taking infinitary objects as arguments is said to be
computable if the output can be computed by a procedure that queries
any legitimate representation of the input. For example, a computable
function $f(x)$ from $\RR$ to $\RR$ is given by an algorithm which,
given the ability to request arbitrarily good rational approximations
to $x$, produces arbitrarily good rational approximations to $y =
f(x)$, in the sense above. In other words, $f$ is given by algorithms
that compute functions $a_y$ and $r_y$ representing $y$, given the
ability to query ``oracles'' $a_x$ and $r_x$ representing $x$.

Similar considerations apply to separable Hilbert spaces, which are
assumed to come with a fixed choice of basis. An element of the space
can be represented by a sequence of finite linear combinations of
basis elements together with a bound on their rate of rate of
convergence in the Hilbert space norm; once again, such an element is
said to be computable if it has a computable representation. The inner
product and norm are then computable operations on the entire space. A
bounded linear operator can be represented by the sequence of values
on elements of the basis, and is computable if that sequence is. In
general, a computable bounded linear operator need not have a
computable norm (see \cite{bishop:bridges:85,avigad:simic:06}).

Computability with respect to a measure space can be understood in
similar ways. A measurable function is represented by a sequence of
suitably simple functions that approximate it in the $L^1$ norm,
together with a rate of convergence. Note that this means that a
measurable function is represented only up to a.e.~equivalence. One
can associate to any measure preserving operator $\tau$ the bounded
linear operator $T f = f \circ \tau$, and take $\tau$ to be
represented by any representative of the associated $T$.

The following theorem shows that it is not always possible to compute
a bound on the rate of convergence of a sequence of ergodic averages
from the initial data.

\begin{theorem}
\label{noncomputability:thm}
There are a computable measure preserving transformation of $[0,1]$
under Lebesgue measure and a computable characteristic function $f =
\chi_A$ such that if $f^* = \lim_n A_n f$, then $\| f^* \|_2$ is not
a computable real number.
  
  In particular, $f^*$ is not a computable element of the Hilbert
  space, and there is no computable bound on the rate of convergence
  of $\la A_n f \ra$ in either the $L^2$ or $L^1$ norm. Nor is there a
  bound on the pointwise rate of convergence of $\la A_n f \ra$, in
  the sense of Theorem~\ref{explicit:pet}.
\end{theorem}

\begin{proof}
  It suffices to prove the assertion in the first sentence. The rest
  of the assertions follow, since if $f^*$ were computable, then $\|
  f^* \|_2$ would be computable, and if there were a computable bound
  on the rate of convergence of $\la A_n f \ra$ in the $L^2$ norm,
  then $f^*$ would be a computable element of $L^2([0,1])$. Computable
  bounds on the rate of convergence in either of the other senses
  mentioned in the remainder of the theorem would imply a computable
  bound on the rate of convergence in the $L^2$ norm.

  To prove the assertion in the first sentence, we use a variant of
  constructions described in \cite{avigad:simic:06,simic:07}. First,
  suppose $f$ is the characteristic function of the interval
  $[0,1/2)$, and $\tau$ is the rotation $\tau x = (x + a) \mod 1$,
  where $a$ is either $0$ or $1/2^j$ for some $j \geq 1$. If $a = 0$,
  then $f^* = f$ and $\| f^* \|^2_2 = 1 / 2$. If $a = 1 / 2^j$ for any
  $j \geq 1$, then $f^*$ is the constant function equal to $1/2$, and
  $\| f^* \|^2_2 = 1 / 4$. Thus knowing $\|f^*\|_2$ allows us to
  determine whether $a = 0$.
  
  Our strategy will be to divide $[0,1)$ into intervals
  $[1-2^i,1-2^{i+1})$, and let $\tau$ rotate each interval by a
  computable real number $a_i$ that depends on whether the $i$th
  Turing machine halts. With a suitable choice of $f$, the limit $f^*$
  of the sequence $\la A_n f \ra$ will then encode information
  as to which Turing machines halt on input $0$.

  The details are as follows. Let $T(e,x,s)$ be Kleene's $T$
  predicate, which asserts that $s$ codes a halting computation sequence
  of Turing machine $e$ on input $x$. The predicate $T$ is computable,
  but the set $\{ e \st \ex s T(e,0,s) \}$ is not. Without loss of
  generality, we can assume that for any $e$ and $x$ there is at most
  one $s$ such that $T(e,x,s)$ holds. We will prove the theorem by
  constructing computable $\tau$ and $f$ such that $\{ e \st \ex s
  T(e,0,s) \}$ is computable from $\| f^* \|_2$.

  Define the computable sequence $\la a_i \ra$ of computable
  reals by setting
\[
  a_i = \left\{
    \begin{array}{ll}
      1/2^{i+j+1} & \mbox{for the unique $j$ satisfying  $T(i,0,j)$,
        if there is one}
      \\
      0 & \mbox{otherwise}
    \end{array}
  \right.
\]
Let $\tau$ be the measure preserving transformation that rotates each
interval $[1-2^i,1-2^{i+1})$ by $a_i$. To see that the sequence
$\la a_i \ra$ is computable, remember that we only need to by
able to compute approximations to the $a_i$'s uniformly; we can
do this by testing $T(i,0,j)$ up to a sufficiently large value of $j$. To
see that $\tau$ is computable, remember that it is sufficient to be
able to compute approximations to the value of $T$ applied to any
simple function, given rational approximations to the the $a_i$'s.

Let $f$ be the characteristic function of the set $\bigcup_i
[1-2^i,1-3 \cdot 2^{i+2})$, so that $f$ is equal to $1$ on the left
half of each interval $[1-2^i, 1-2^{i+1})$ and $0$ on the right half.
Let $f^* = \lim_n A_n f$. Then
\[
\| f^* \|^2_2 = 
\sum_{\{i \st \ex j T(i,0,j)\}} \frac{1}{4} \cdot \frac{1}{2^{i+1}} + 
\sum_{\{i \st \lnot \ex j T(i,0,j)\}} \frac{1}{2}\cdot \frac{1}{2^{i+1}}
\]
and
\[
\frac{1}{2} - \| f^* \|^2_2 = \sum_{i \in \NN} \frac{1}{2} \cdot
\frac{1}{2^{i+1}}  - \| f^* \|^2_2 = \sum_{\{i \st \ex j T(i,0,j)\}}
\frac{1}{2^{i+3}}.
\]
Calling this last expression $r$, it suffices to show that $\{ i \st
\ex j T(i,0,j) \}$ is computable from $r$. But the argument is now
standard (see \cite[Section 0.2, Corollary 2a]{pourel:richards:89} or
\cite[Theorem III.2.2]{simpson:99}). For each $n$, let
\[
r_n = \sum_{\{i \st \ex {j \leq n} T(i,0,j)\}} \frac{1}{2^{i+3}}.
\]
Then the sequence $\la r_n \ra$ is computable and increases
monotonically to $r$. To determine whether Turing machine $i$ halts on
input $0$, it suffices to search for an $n$ and an approximation to
$r$ sufficiently good to ensure $| r - r_n | < 1/2^{i+3}$. Then we
only need to check if there is a $j < n$ such that $T(i,0,j)$ holds;
if there isn't, $T(i,0,j)$ is false for every $j$.
\end{proof}

The proof of Theorem~\ref{noncomputability:thm} relied on the fact
that the system we constructed is not ergodic; we used the behavior of
the system on each ergodic component to encode the behavior of a
Turing machine. The next two theorems and their corollary show that if,
on the other hand, the space in question is ergodic, then one always
has a computable rate of convergence.

\begin{theorem}
\label{computability:thm}
Let $T$ be a nonexpansive linear operator on a separable Hilbert space
and let $f$ be an element of that space. Let $f^* = \lim_n A_n
f$. Then $f^*$, and a bound on the rate of convergence of $\la A_n f
\ra$ in the Hilbert space norm, can be computed from $f$, $T$, and $\|
f^* \|$.
\end{theorem}

\begin{proof}
  It suffices to show that one can compute a bound on the rate of
  convergence of $\la A_n f \ra$ from the given data. Assuming $f$ is
  not already a fixed point of $T$, write $f = f^* + g$, and let the
  sequences $\la g_i \ra$, $\la u_i \ra$, and $\la a_i \ra$ be defined
  as in Section~\ref{constructive:met:section}. Then $g = \lim_i g_i$,
  and $g_i = u_i - Tu_i$ and $a_i = \| g_i \|$ for every $i$. Let $a =
  \lim_i a_i$. Then $a = \| g \| = \sqrt{\|f\|^2 - \|f^* \|^2}$ can
  be computed from $f$ and $\| f^* \|_2$. For any $m$, $n \geq m$, and
  $i$, we have
\begin{equation*}
\begin{split}
  \| A_m f - A_n f \| & =  \| A_m g - A_n g \| \\
       & \leq \| A_m g_i - A_n g_i \| + \| A_m (g - g_i) \| + 
         \| A_n (g - g_i) \| \\
       & \leq \| A_m g_i \| + \| A_n g_i \| + 2\| g - g_i \| \\
       & \leq \| A_m g_i \| + \| A_n g_i \| + 2 \sqrt{2(a - a_i) \| f
         \| }
\end{split}
\end{equation*}
as in the proof of Lemma~\ref{aigi}. Given $\varepsilon$,
using the given data we can now find an $i$ such that the last term on
the right hand side is less than $\varepsilon / 2$, compute $u_i$, and
then, using Lemma~\ref{4facts}.2, determine an $m$ large enough so that
for any $n \geq m$, $\| A_m g_i \| + \| A_n g_i \| < \varepsilon / 2$.
\end{proof}

\begin{theorem}
  Let $\mdl X = \la X, \mdl B, \mu \ra$ be a separable measure
  space, let $\tau$ be a measure preserving transformation of $\mdl
  X$, and let $T$ be the associated Koopman operator. Then for any $f$
  in $L^2(\mdl X)$, bounds on the rate of convergence in the $L^2$
  norm, in the $L^1$ norm, and in the sense of
  Theorem~\ref{explicit:pet} can be computed from $f$, $T$, and $\|
  f^* \|_2$.
\end{theorem}

\begin{proof}
  The previous theorem provides bounds on the rate of convergence in
  the $L^2$ norm, and hence in the $L^1$ norm as well. 

  For convergence in the sense of Theorem~\ref{explicit:pet}, consider
  inequality (\ref{pointwise:calc}), where now $h$ is $f^*$ and $f_i$ is 
  $g_i + f^*$:
\begin{equation*}
\begin{split}
  | A_m f(x) - A_n f(x) |  & \leq | A_m(f - (g_i + f^*))(x)| \\
  & \quad \quad + |A_m g_i(x) | + | A_n g_i(x) | + |A_n (f - (g_i +
  f^*))(x) |.
\end{split}
\end{equation*}
The sequence $\la f - (g_i + f^*)\ra_{i \in \NN}$ converges to $0$ in
the $L^2$ norm, and, as in the proof of
Theorem~\ref{computability:thm}, we can compute a bound on the rate of
convergence from the given data. Using Corollary~\ref{maximal:cor} we
can make the first and last terms small outside of a small set of
exceptions, independent of $m$ and $n$, by making $i$ sufficiently
large. Using Lemma~\ref{other:terms:lemma} we can then determine how
large $n$ has to be so that the remaining terms are small outside a
small set of exceptions, for all $m > n$.
\end{proof}

\begin{corollary}
  With $\mdl X = \la X, \mdl B, \mu \ra$ as above, suppose $\tau$ is
  an ergodic measure preserving transformation. Then for any $f$
  in $L^2(\mdl X)$, bounds on the rate of convergence in the $L^2$
  norm, in the $L^1$ norm, and in the sense of
  Theorem~\ref{explicit:pet} can be computed from $f$, $T$, and $\mu$.
  
  For any $f$ in $L^1(\mdl X)$, bounds on the rate of convergence in
  the $L^1$ norm and in the sense of Theorem~\ref{explicit:pet} can be
  computed from $T$, $\mu$, and a sequence of $L^2(\mdl X)$ functions
  approximating $f$ in the $L^1$ norm (together with a rate of
  convergence).
\end{corollary}

\begin{proof}
  If the system is ergodic, $f^*$ is a.e.~equal to the constant $\int
  f \; d\mu$, in which case $\| f^* \|_2 = | \int f \; d\mu |$. Thus
  $\| f^* \|_2$ is computable from $f$ and $\mu$, and we can apply the
  previous theorem.

  Suppose now $f \in L^1$ and we are given a sequence $\la f_i \ra$ of
  $L^2$ functions approaching $f$ in the $L^1$ norm, together with a
  rate of convergence. Since 
\[
  \| A_n f - A_n f_i \|_1 = \| A_n (f - f_i) \|_1 \leq \| f - f_i \|_1
\]
 for every $n$, we can make $\| A_m f - A_n f \|_1$ small by
  first picking $i$ large enough and then ensuring that $\| A_m f_i -
  A_n f_i \|_1$ is small. Similarly, we can make $| A_m f(x) - A_n
  f(x) |$ small outside a small set of exceptions by first choosing
  $i$ sufficiently large, applying Corollary~\ref{maximal:cor}, and
  then using the previous theorem to choose $n$ large enough so that
  $| A_m f_i(x) - A_n f_i(x) |$ is small outside a small set of
  exceptions for every $m > n$. 
\end{proof}

The issues raised here can be considered from a spectral standpoint as
well. If $T$ is a unitary transformation of a Hilbert space, then the
spectral measure $\sigma_f$ associated to $f$ can be described in the
following way. For each $k \in \ZZ$, let $b_k =\lip T^k f, f \ra$ be
the $k$th autocorrelation coefficient of $f$. Let $\TT$ be the circle
with radius $1$, identified with the interval $[0,2 \pi)$. Let $I$ be
the linear operator on the complex Hilbert space $L^\CC_2(\TT)$
defined with respect to the basis $\lip e^{i k \theta} \ra_{k \in
  \ZZ}$ by $I(e^{i k \theta}) = b_k$. The sequence $b_k$ is a positive
definite sequence, and so by Bochner's theorem (see
\cite{katznelson:04,nadkarni:98}), there is a positive measure
$\sigma_f$ on $\TT$ such that $I(g) = \int g \; d\sigma_f$. It is well
known that $\| f^* \|_2^2 = \sigma_f(\{0\})$, and Kachurovskii
\cite[page 670]{kachurovskii:96} shows that if $f^* = 0$, then for
every $n$ and $\delta \in (0, \pi)$,
\[
\| A_n f \|_2 \leq \sqrt{\sigma_f(-\delta,\delta)} + \frac{4 \| f
  \|_2}{n \sin (\delta/2)}.
\]
This last expression shows that, in the case where $f^* = 0$, one can
compute a bound the rate of convergence of $\la A_n f \ra$ from a
bound on the rate of convergence of $\sigma_f(-\delta, \delta)$ as
$\delta$ approaches $0$. The problem is that $I$ is not necessarily a
bounded linear transformation, and so $\sigma_f$ is not generally
computable from $f$. Theorem~\ref{computability:thm} above shows that
for any $f$ it is nonetheless possible to compute $f^*$ from
$\sigma_f(\{0\})$, $f$, and $T$.

  For any set of natural numbers $X$, let $X'$ denote the halting
  problem relative to $X$. The proof of
  Theorem~\ref{noncomputability:thm} shows, more generally, the
  following:
\begin{theorem}
  For any set of natural numbers $X$, there are a Lebesgue-measure
  preserving transformation $\tau$ of $[0,1]$, computable from $X$,
  and a computable element $f$ of $L^2([0,1])$, such that $X'$ is
  computable from $\| f^* \|_2$.
\end{theorem}

The results in this section can be adapted to yield information with
respect to provability in restricted axiomatic frameworks.
Constructive mathematics, for example, aims to use only principles that
can be given a direct computational interpretation (see, for example,
\cite{bishop:67,bishop:bridges:85}). There is also a long tradition of
developing mathematics in classical theories that are significantly
weaker than set theory. In the field of reverse mathematics,
this is done with an eye towards calibrating the degree of
nonconstructivity of various theorems of mathematics (see
\cite{simpson:99}); in the field of proof mining, this is done with an
eye towards mining proofs for additional information (see
Section~\ref{proof:mining:section}, below).

When a theorem of modern mathematics is not constructively valid, one
can search for an ``equal hypothesis'' substitute, i.e.~a constructive
theorem with the same hypotheses, and with a conclusion that is easily
seen to be classically equivalent to the original theorem. Bishop's
upcrossing inequalities, as well as the results of
Spitters~\cite{spitters:06}, are of this form. The results of
Sections~\ref{constructive:met:section} and \ref{pointwise:section}
are also of this form, and are provable both constructively and in the
weak base theory $\na{RCA_0}$ of reverse mathematics. One can also
look for ``equal conclusion'' substitutes, by seeking classically
equivalent but constructively stronger
hypotheses. Theorem~\ref{computability:thm} has this flavor, but it is
hard to see how one can turn it into a constructive theorem, because
it is not clear how one can refer to $\| f^* \|_2$ without
presupposing that $\la A_n f \ra$ converges.  One can show,
constructively and in $\na{RCA_0}$, that if the projection of $f$ on
the subspace $N$ described in the proof of Theorem~\ref{met} exists
then $\la A_n f \ra$ converges; but the assumption that the projection
of $f$ on $M$ exists is not sufficient (see
\cite{avigad:simic:06,spitters:06} and the corrigendum to the
latter). An interesting equal conclusion constructive version of the
pointwise ergodic theorem can be found in Nuber \cite{nuber:72}.

Theorem~\ref{noncomputability:thm} shows that the mean and pointwise
ergodic theorems do not have constructive proofs. In fact, in the
setting of reverse mathematics, they are equivalent to a set-existence
principle known as arithmetic comprehension over $\na{RCA_0}$. For
stronger results, see \cite{avigad:simic:06,simic:07}.


\section{Proof-theoretic techniques}
\label{proof:mining:section}

The methods we have used in Sections~\ref{constructive:met:section}
and \ref{pointwise:section} belong to a branch of mathematical logic
called ``proof mining,'' where the aim is to develop general
techniques that allow one to extract additional information from
nonconstructive or ineffective mathematical proofs. The program is
based on two simple observations: first, ordinary mathematical proofs
can typically be represented in formal systems that are much weaker
then axiomatic set theory; and, second, proof theory provides general
methods of analyzing formal proofs in such theories, with an eye
towards locating their constructive content. Traditional research has
aimed to show that many classical theories can be reduced to
constructive theories, at least in principle, and has developed a
variety of techniques for establishing such reductions. These include
double-negation translations, cut-elimination, Herbrand's theorem,
realizability, and functional interpretations. (The \emph{Handbook of
  Proof Theory} \cite{buss:98} provides an overview of the range of
methods.) Proof mining involves adapting and specializing these
techniques to specific mathematical domains where additional
information can fruitfully be sought.

Our constructive versions of the mean and pointwise ergodic theorems
are examples of Kreisel's no-counterexample interpretation
\cite{kreisel:51,kreisel:59alt}. Effective proofs of such translated
statements can often be obtained using variants of G\"odel's
functional (``Dialectica'') interpretation \cite{goedel:58} (see also
\cite{avigad:feferman:98}). Ulrich Kohlenbach has shown that the
Dialectica interpretation can be used as an effective tool; see, for
example, \cite{kohlenbach:05,kohlenbach:unp:b}. For example, our
constructive mean ergodic theorem,
Theorem~\ref{constructive:met:final}, provides bounds that depend only
on $K$ and $\| f \| / \varepsilon$. In fact, the usual proofs of the
mean ergodic theorem can be carried out in axiomatic frameworks for
which the general metamathematical results of Gerhardy and Kohlenbach
\cite{gerhardy:kohlenbach:05} guarantee such uniformity.

While the methods of the paper just cited do show how one can find an
explicit expression for the requisite bound, the resulting expression
would not yield, a priori, useful bounds. For that, a more refined
analysis, due to Kohlenbach \cite{kohlenbach:97a}, can be used. The
nonconstructive content of the Riesz proof of the mean ergodic theorem
can be traced to the use of the principle of convergence for bounded
monotone sequences of real numbers. In formal symbolic terms, the fact
that every bounded increasing sequence of real numbers converges can
be expressed as follows:
\begin{multline*}
\forall {a : \NN \to \RR, c \in \RR}
(\fa i (a_i \leq a_{i+1}
\leq c) \rightarrow \\
\fa {\varepsilon > 0} \ex n \fa {m \geq
n} (|a_m - a_n| \leq \varepsilon) ).
\end{multline*}
Using a principle known as ``arithmetic comprehension,'' we can
conclude that there is a function, $r$, bounding the rate of
convergence:
\begin{multline}
\label{convergence:eq}
\fa {a : \NN \to \RR, c \in \RR}
(\fa i (a_i \leq a_{i+1}
\leq c) \rightarrow \\
\ex r \fa {\varepsilon > 0} \fa {m \geq
r(\varepsilon)} (|a_m - a_{r(\varepsilon)} | \leq \varepsilon) ).
\end{multline}
In general, $r$ cannot be computed from the sequence $\la a_i \ra$. On
the other hand, the proof of Theorem~\ref{computability:thm} shows
that witnesses to the mean ergodic theorem can be computed from a
bound $r$ on the rate of convergence, for a sequence $\la a_i \ra$
that is explicitly computed from $T$ and $f$. Moreover, the proof of
this fact can be carried out in a weak theory.  Kohlenbach's results
show that, in such situations, one can compute explicit witnesses to
the Dialectica translation to the theorem in question from a weaker
version of principle (\ref{convergence:eq}):
\begin{multline*}
\fa {a : \NN \to \RR, c \in \RR}
( \fa i (a_i \leq a_{i+1}
\leq c) \rightarrow \\ 
\forall {\varepsilon > 0, M} \ex n (M(n) \geq n
\limplies  (|a_{M(n)} - a_n| \leq \varepsilon) ).
\end{multline*}
This last principle can be given a clear computational interpretation:
given $\varepsilon$ and $M$, one can iteratively compute $0, M(0),
M(M(0)), \ldots$ until one finds a value of $n$ such that $|a_{M(n)} -
a_n| \leq \varepsilon$. This information can then be used to witness
the Dialectica translation of the conclusion, that is, our
constructive mean ergodic theorem.

This strategy is clearly in evidence in
Section~\ref{constructive:met:section}. In practice, it is both
infeasible and unnecessary to express the initial proof in completely
formal terms. Rather, one undertakes a good deal of heuristic
manipulation of the original proof, using the translation to determine
what form intermediate lemmas should have, and how they should be
combined. The metamathematical results are therefore used as a guide,
providing both guarantees as to what results can be achieved, and the
strategies needed to achieve them.

 

\end{document}